\documentclass[11pt]{amsart}
\usepackage{amsmath, amsthm, amssymb}
\usepackage{graphicx}
\usepackage{mathbbol}
\usepackage{txfonts}
\usepackage[usenames]{color}

\newtheorem{thm}{Theorem}[section]
\newcommand{\bt}{\begin{thm}}
\newcommand{\et}{\end{thm}}

\newtheorem{ex}[thm]{Example}

\newtheorem{cor}[thm]{Corollary}   

\newcommand{\bc}{\begin{cor}}
\newcommand{\ec}{\end{cor}}

\newtheorem{lem}[thm]{Lemma}   

\newcommand{\bl}{\begin{lem}}
\newcommand{\el}{\end{lem}}

\newtheorem{prop}[thm]{Proposition}
\newcommand{\bp}{\begin{prop}}
\newcommand{\ep}{\end{prop}}

\newtheorem{defn}[thm]{Definition}
\newtheorem{conj}[thm]{Conjecture}

\newcommand{\bd}{\begin{defn}}    
\newcommand{\ed}{\end{defn}}

\newtheorem{rmrk}[thm]{Remark}   

\newcommand{\br}{\begin{rmrk}}
\newcommand{\er}{\end{rmrk}}

\newtheorem{example}[thm]{Example}

\newcommand{\grad}{\nabla}

\newcommand{\be}{\begin{equation}}
\newcommand{\ee}{\end{equation}}

\newcommand{\N}{\mathbb{N}}

\newcommand{\R}{\mathbb{R}}

\newcommand{\E}{\mathbb{E}}

\newcommand{\diam}{\operatorname{diam}}

\newcommand{\disjointunion}{\sqcup}

\newcommand{\vol}{\operatorname{Vol}}

\def\rr{\mathbb{R}}

\def\rmin{r_{\mathrm{min}}}

\def\iff{\Longleftrightarrow}

\def\too{\longrightarrow}

\def\madm{\mathrm{m}_{\mathrm{ADM}}}
\def\rmin{r_{\textrm{min}}}

\DeclareMathOperator{\RS}{RotSym}

\begin{document}

\title[Stability of the Positive Mass Theorem]{Stability of the Positive Mass Theorem 
for Rotationally Symmetric Riemannian Manifolds}

\author{Dan A. Lee}
\thanks{Lee is partially supported by a PSC CUNY Research Grant and NSF DMS \#0903467.}
\address{CUNY Graduate Center and Queens College}
\email{dan.lee@qc.cuny.edu}

\author{Christina Sormani}
\thanks{Sormani is partially supported by a PSC CUNY Research Grant and NSF DMS \#1006059.}
\address{CUNY Graduate Center and Lehman College}
\email{sormanic@member.ams.org}

\date{}


\begin{abstract}
We study the stability of the Positive Mass Theorem 
using the Intrinsic Flat Distance.
In particular 
we consider the class of complete asymptotically flat rotationally symmetric
Riemannian manifolds with nonnegative scalar curvature and no interior closed minimal surfaces whose boundaries are either outermost minimal hypersurfaces or are empty.  
We prove that a sequence of these manifolds whose ADM masses converge to zero must converge to Euclidean space in the pointed Intrinsic Flat sense.  In fact
we provide explicit bounds on the Intrinsic Flat Distance between annular regions in the manifold and annular regions in Euclidean space by
constructing an explicit filling manifold and estimating its volume.
In addition, we include a variety of propositions
that can be used to estimate the Intrinsic Flat
distance between Riemannian manifolds without
rotationally symmetry.   Conjectures
regarding the Intrinsic Flat stability of the Positive Mass Theorem
in the general case are proposed in the final section.
\end{abstract}

\maketitle

\newpage

\section{Introduction}

The Positive Mass Theorem states that any complete asymptotically flat manifold of nonnegative scalar curvature has nonnegative ADM mass.   
Furthermore, if the ADM mass is zero, then the manifold must be Euclidean space.  The second statement may be thought of as a rigidity theorem, and it is natural to consider the \emph{stability} of this rigidity statement.  That is, if the ADM mass is small, 
in what sense can we say that the manifold is ``close'' to Euclidean space?  This is known to be a subtle question for many reasons.  

The ADM mass was defined by Arnowitt-Deser-Misner in
\cite{ADM-mass} and the Positive Mass Theorem was first proven in the
rotationally symmetric case by
physicists Jang, Leibovitz and Misner in
\cite{Jang-positive-energy}
\cite{Leibovitz-1970} \cite{Misner-astro}.  The general three dimensional
case was proven by Schoen-Yau, and later by Witten
using spinors
\cite{Schoen-Yau-positive-mass} \cite{Witten-positive-mass}.  
Schoen and Yau's proof generalizes to dimensions $< 8$
(see \cite{Schoen-1989})
using Bartnik's higher dimensional ADM mass \cite{Bartnik-1986}, while Witten's proof holds on all spin manifolds (c.f. \cite{Schoen-1989}).

The problem of stability for the Positive Mass Theorem has been studied by the first author in \cite{Lee-near-equality}, by Finster with Bray and Kath
in \cite{Bray-Finster}, \cite{Finster-Kath} \cite{Finster} and 
by Corvino in \cite{Corvino}.  
The work of Finster and his collaborators mainly focuses on using the ADM mass to obtain $L^2$ bounds on curvature.  Corvino proves that
with uniform bounds on sectional curvature, a manifold with small
enough ADM mass is diffeomorphic to Euclidean space.
The present work complements the results of \cite{Lee-near-equality}.   That article dealt with convergence to Euclidean space outside some compact set.   In this paper we tackle the much harder problem of trying to understand what happens inside the compact set.  We place no assumptions on sectional
curvature, so it is possible for the manifolds to have boundary inside
the compact region.  Because we expect the general problem to be difficult, we focus on the simple case of rotationally symmetric manifolds and state a more general conjecture at the end of the paper.

One serious concern is that even if the ADM mass is small, there can be arbitrarily deep ``gravity wells.''  See Figure \ref{fig-thm-pted}
and Example~\ref{ex-deep-well}.  As the ADM mass approaches zero, these deep gravity wells do not converge to Euclidean space in any conventional sense, including the pointed Gromov-Hausdorff sense.  For this reason, we turn to the Intrinsic Flat Distance between
Riemannian manifolds, a notion developed by the second author and S.\ Wenger which can be controlled using
volumes and filling volumes \cite{SorWen2}.

\begin{figure}[h] 
   \centering
   \includegraphics[width=5in]{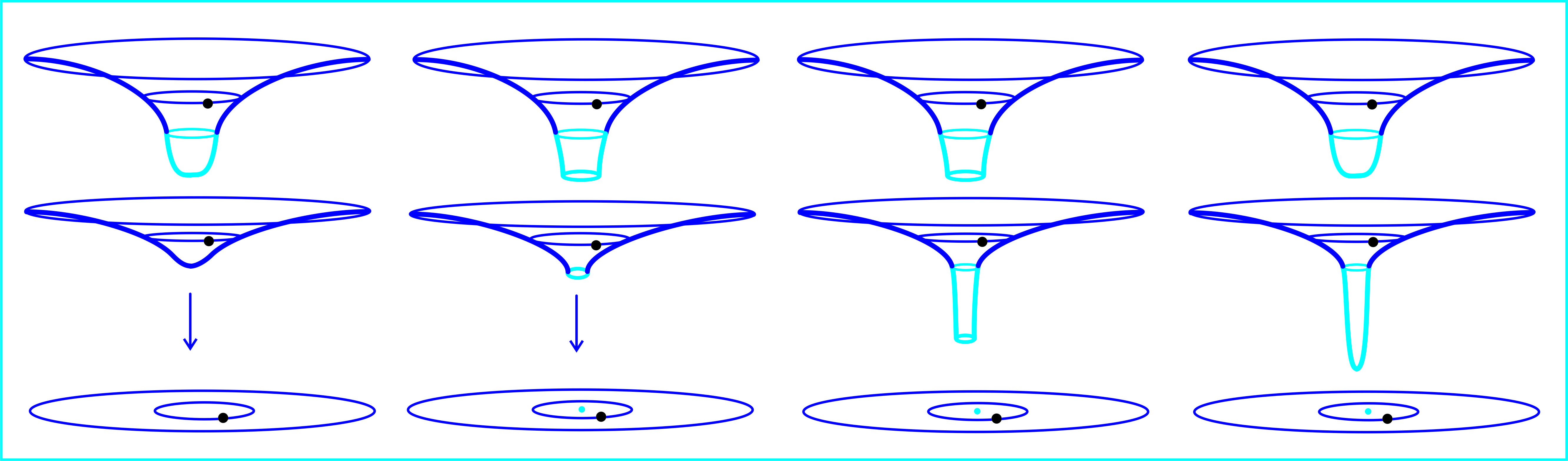} 
   \caption{Four sequences of asymptotically flat manifolds of nonnegative scalar curvature whose ADM masses converge to zero.
      }
   \label{fig-thm-pted}
\end{figure}

The Intrinsic Flat Distance is defined and studied in \cite{SorWen2} by applying sophisticated ideas of Ambrosio-Kirchheim \cite{AK}
extending earlier work of Federer-Fleming \cite{FF} and
Whiney\cite{Whitney}.  While the definition of the Intrinsic Flat Distance involves abstract metric spaces and geometric measure theory, for the present work we can restrict our attention to Riemannian manifolds.   Given two compact orientable Riemannian manifolds $M^m_1$ and $M^m_2$ with boundary, and metric isometric embeddings $\psi_i: M_i \to Z$ into some Riemannian
manifold (possibly piecewise smooth with corners), $Z$, 
an upper bound for the Intrinsic Flat Distance is attained as follows:
\be \label{eqn-def-intrinsic-flat-1}
d_{\mathcal{F}}(M^m_1, M^m_2) \le \vol_{m+1}(B^{m+1}) +\vol_m(A^m)
\ee
where $B^{m+1}$ is an oriented region in $Z$ and $A^m$ is
defined so that
\be \label{eqn-def-intrinsic-flat-2}
\int_{\psi_1(M_1)}\omega -\int_{\psi_2(M_2)}\omega=\int_{\partial B}\omega + \int_A\omega 
\ee
for any $m$ differential form $\omega$ on $Z$.   We call
$B^{m+1}$ a {\em filling manifold} between $M_1$ and $M_2$
and $A^m$ the {\em excess boundary}.
A metric isometric embedding, 
$\psi: M \to Z$ is a map such that
\be\label{eqn-def-metric-isom-embed}
d_Z(\psi(x), \psi(y)) = d_M(x,y) \qquad \forall x,y \in M.
\ee
This is significantly stronger than a Riemannian isometric embedding
which preserves only the Riemannian structure and thus lengths
of curves but not distances between points as in 
(\ref{eqn-def-metric-isom-embed}).

Our main results concern rotationally symmetric Riemannian manifolds
of dimension $3$ and up:

\begin{defn}  \label{def-rot-sym}
Given $m\ge 3$,
let $\RS_m$ be the class of complete $m$-dimensional 
rotationally symmetric Riemannian manifolds
of nonnegative scalar curvature with no closed interior
minimal hypersurfaces which either have no boundary or have a boundary
which is a stable minimal hypersurface.
\end{defn}

This class of spaces includes the classical rotationally symmetric
gravity wells and Schwarzschild spaces.   The boundary, when it
exists, is called the ``apparent horizon" of a black hole.  
Nonnegative scalar curvature may be viewed as a 
physical notion of nonnegative mass density.
The ADM mass of such a manifold exists when it is
asymptotically flat and intuitively records the total mass of the space
as a physical system.   We review the scalar curvature and ADM mass
of manifolds in this class within the paper.  

The condition regarding minimal hypersurfaces is included here (just
as it is in the Penrose Inequality) because
complicated geometry can ``hide'' behind a minimal hypersurface
without affecting the ADM mass (c.f. \cite{Gibbons} \cite{Huisken-Ilmanen}).
Note that we need not explicitly assume asymptotic flatness here 
because finite ADM mass in $\RS_m$ implies asymptotic flatness.

\begin{thm}\label{thm-main}   
Given any $\epsilon>0$, $D>0$, $A_0>0$, $m\in \N$
there exists a $\delta=\delta(\epsilon, D, A_0, m)>0$
such that if $M^m\in\RS_m$ has ADM mass 
$\madm(M)<\delta$ and 
$\E^m$ is Euclidean space of the same dimension,
then
\be
 d_{\mathcal{F}}(\,T_D(\Sigma_0)\subset M^m\,,\, T_D(\Sigma_0)\subset \E^m\,)\,<\, 
 \epsilon. 
 \ee
 where  $\Sigma_0$ is the symmetric sphere of area  
$\vol_{m-1}(\Sigma_0)=A_0$, and $T_D(\Sigma)$ is the 
tubular neighborhood of radius $D$ around $\Sigma_0$.   
 \end{thm}
 
See Remark~\ref{rmrk-scaling} concerning the fact that the
flat distance does not scale with the metric
on the manifolds.   In the proof precise estimates
on $\delta(\epsilon, A_0, D, m)$ are provided.
 
Applying Ambrosio-Kirchheim's Slicing Theorem as in \cite{SorWen1},
we then have the following immediate corollary:

\begin{cor}\label{thm-pted}
Let $M_j^m$ be a
sequence in $\RS_m$. 
Fix an area $A_0$, and choose $p_j \in \Sigma_j$ 
to lie on the symmetric sphere $\Sigma_j\subset M^m_j$
of area $\vol_{m-1}(\Sigma_j)=A_0$.  
If $\mathrm{m}_{\mathrm{ADM}}(M_j)$ converges to 
$0$
then $(M_j^m, p_j)$ converges to Euclidean space $(\E^m,0)$ 
in the pointed intrinsic flat sense.  That is, for almost any $D>0$
there exists $D_j \to D$ such that
$B_{p_j}(D_j)\subset M_j^m$ converges in the intrinsic flat sense to $B_0(D)\subset\E^m$.
\end{cor}

Throughout the paper, we provide techniques which can be used in a more
general setting to bound the Intrinsic Flat Distance using Riemannian
methods rather than Geometric Measure Theory.   These may be
applied to solve some of the open problems in our final section or
even problems which do not involve scalar curvature.

In Section~\ref{Sect-Positive} we review rotationally symmetric manifolds with nonegative scalar curvature, the monotonicity of Hawking mass 
and the definition of ADM mass.  There we describe a well known
Riemannian isometric embedding of these manifolds as
graphs in Euclidean space [Lemma~\ref{lem-graph}] and
review the monotonicity of Hawking mass [Theorem~\ref{thm-monotonicity}].
We control the diameter of the boundary in terms of the
ADM mass [Lemma~\ref{lem-rmin}] and the slope of the
graph [Lemma~\ref{lem-F'}] in terms of the ADM mass.
We conclude the section with classical rotationally symmetric 
examples including those 
depicted in Figure~\ref{fig-thm-pted} [Examples~\ref{ex-Sch-to-0} 
and~\ref{ex-deep-well}].

In Section~\ref{Sect-Embed} we prove a variety of propositions
about metric isometric embeddings and estimates on the intrinsic
flat distance.  This includes Theorem~\ref{thm-warp} using warped products to
construct metric isometric embeddings and Theorem~\ref{thm-convex}  regarding
the construction of metric isometric embeddings from convex embeddings.
The later theorem may be useful to those studying quasilocal mass.
Theorem~\ref{thm-embed-const} provides a general method for constructing a metric
isometric embedding from a Riemannian isometric embedding
using an ``embedding constant".  This theorem is applied to bound
the Intrinsic Flat Distance as a function of the embedding constant
in Propositions~\ref{embed-const} 
and~\ref{embed-const-2}.  Theorem~\ref{thm-Z}
provides a bound on the embedding constant when the Riemannian
isometric embedding is a graph over a manifold with boundary.
See also Remark~\ref{rmrk-embed-const}.

In Section~\ref{Sect-Pos} we
prove Theorem~\ref{thm-main}. 
See Figure~\ref{fig-flat-filling} for a depiction of the explicit
filling manifold and excess boundary used in the
proof.  Lemma~\ref{lem-well} determines where to
cut off a possibly deep well in the estimate.  Then 
the earlier theorems and lemmas are applied to
prove we have metric isometric embeddings and
to estimate the volumes.

In Section~\ref{Sect-Gromov} we review Gromov-Hausdorff
convergence.  We provide a new method for estimating
the Gromov-Hausdorff distance using embedding constants,
[Propositions~\ref{embed-const-3} and~\ref{embed-const-4}]
and apply them to construct explicit examples demonstrating
that even with an assumption on rotational symmetry,
the Positive Mass Theorem is not stable with respect to the
Gromov-Hausdorff distance [Example~\ref{ex-not-GH}] due
to the existence of thin deep wells.   This section closes with
an example of
a sequence of 3 dimensional manifolds
with positive scalar curvature with no rotational symmetry
whose ADM mass converges to
0 but has no subsequence converging in the Gromov-Hausdorff
sense to any space due to the existence of an increasingly
dense collection of wells [Example~\ref{ex-noncompact}].  Nevertheless
this sequence converges in the Intrinsic Flat sense to Euclidean space.

In Section~\ref{Sect-Open} we discuss the general
question of asymptotically flat Riemannian 
manifolds, $M^m$, of positive scalar curvature with no interior
minimal surfaces that either have an outermost
minimizing boundary or no boundary.  
We close the paper with conjectures and open
problems concerning various subclasses of such manifolds and
the stability of the Positive Mass Theorem for those subclasses.  We
hope that some of our more general theorems regarding the
Intrinsic Flat Distance will prove useful to those attempting these
problems.

The authors would like to thank Jim Isenberg and
Jack Lee for organizing the Pacific Northwest
Geometry seminar and for requesting a collection of open problems.
The first author would like to thank Hubert Bray for various thought-provoking conversations on the near equality cases of the Positive Mass Theorem.    The second author would like to thank 
Tom Ilmanen for recommending the development of a new convergence 
to handle problems involving scalar curvature many years ago, 
Jeff Cheeger for requesting a section be included to illucidate why
Gromov-Hausdorff convergence is unsuited for these problems and 
Lars Andersson for his recent suggestion of a need for a scalable
Intrinsic Flat Distance which lead to the final open problem
listed in this paper.

\section{Positive Scalar Curvature, ADM Mass and Asymptotic Flatness}
\label{Sect-Positive}

In the first subsection we briefly review the properties of the
manifolds in $\RS_m$ and  
the key formulas defining their ADM mass.
In the next subsection we embed the manifold into 
Euclidean space as a graph and review the Positive Mass Theorem and 
the monotonicity of the Hawking mass.  In the third
subsection we explore geometric implications of
having a small ADM mass proving key lemmas
which will be applied later to prove the stability of the
positive mass theorem.    We close with a subsection providing 
key rotationally symmetric examples.

\subsection{Setting}\label{ss-setting}

In this paper we consider manifolds $(M^m,g)\in \RS_m$
defined in Definition~\ref{def-rot-sym}.
Since such a manifold is rotationally symmetric 
we can write its metric in
geodesic coordinates, as $g=ds^2 + f(s)^2 g_0$ for some function 
$f:[0,\infty) \to [0,\infty)$
where $g_0$ is the standard metric on the $(m-1)$-sphere
and $s$ is either the distance from the pole, $p_0$, or from the boundary, 
$\partial M$.   

Let $\Sigma'_s$ be a level set of this distance function at a distance
$s$ from the pole or boundary.
We then have the following formulae for the ``area" and mean curvature  
of $\Sigma'_s$:
\be \label{A-s}
\mathrm{A}(s) =\vol_{m-1}(\Sigma_s)=\omega_{m-1} f^{m-1}(s)
\ee
\be\label{H-s}
\mathrm{H}(s)  = \frac{(m-1)f'(s)}{f(s)}
\ee
Thus $\Sigma'_s$ provide a CMC foliation of the manifold.   

Let $r_{min}=f(0)$.   When $\partial M=\emptyset$ then $f(0)=r_{min}=0$ and $f(s)\ge f(0)$ by smoothness at the pole.
When $\partial M \neq \emptyset$ the definition of
$\RS$ states that  $\partial M$ is a stable
minimal surface so $f(0)=r_{min}>0$ and $f'(0)=0$ and $f(s)\ge f(0)$
in that case as well. 

The definition of $\RS$
also requires that $M^m$ has no interior minimal surfaces,
so by (\ref{H-s}), we have
\be
f'(s)\neq 0 \qquad \forall s\in (0,\infty].  
\ee
By the Mean Value Theorem, we see that
\be\label{eqn-f'>0}
f'(s)>0 \qquad \forall s\in (0,\infty].  
\ee
Thus $A(s)$ is increasing and we can uniquely define our
rotationally symmetric constant mean curvature spheres
\be
\Sigma_{\alpha_0}=\Sigma'_{s_0} \textrm{ such that } 
\vol_{m-1}(\Sigma'_{s_0})=\alpha_0.
\ee
Observe that intrinsically these are round spheres of
diameter:
\be
\diam(\Sigma_{\alpha_0})= \pi f (s_0).
\ee

At a point $p\in \partial B_{p_0}(s)$ the scalar curvature is
\be \label{eqn-18}
\mathrm{R} = \frac{m-1}{f^2(s)}\left((m-2)[1- (f'(s))^2] - 2f(s)f''(s)\right) >0.
\ee

Recall the definition of the Hawking mass of a surface, $\Sigma$
in three dimensional manifold:
\be
 \mathrm{m}_{\mathrm{H}}(\Sigma) = \frac{1}{2}\left(\frac{A}{\omega_{2}}\right)\left(1-\frac{1}{4\pi}\int_\Sigma \left(\frac{H}{2}\right)^{2}\right).
 \ee
We define a natural Hawking mass function, $\mathrm{m}_{\mathrm{H}}(s)$, for $M^m\in \RS$
such that in dimension three $\mathrm{m}_{\mathrm{H}}(s)=\mathrm{m}_{\mathrm{H}}(\Sigma_s)$:
\be \label{eqn-hawking-1}
\mathrm{m}_{\mathrm{H}}(s) = \frac{f^{m-2}(s)}{2}(1-(f'(s))^2). 
\ee
Applying (\ref{eqn-18}), we see that
\be \label{eqn-hawking-2}
\mathrm{m}_{\mathrm{H}}'(s) = \frac{f^{m-1}(s)f'(s)}{2(m-1)} \mathrm{R}
\ee
Since we are studying manifolds with $f'(s)>0$ for $s\in (0,\infty)$
and $R\ge 0$, we have the monotonicity of the Hawking mass:
\be\label{eqn-hawking-increases}
\mathrm{m}_{\mathrm{H}}'(s) \ge 0. 
\ee
Observe that when
$\partial M \neq \emptyset$, 
\be \label{eqn-rmin-1}
\mathrm{m}_{\mathrm{H}}(0)=r_{min}^{m-2}/2.
\ee 
This also holds true when $\partial M =\emptyset$, since $\mathrm{m}_{\mathrm{H}}(0)=0$.

We define the
ADM mass of $M^m$ is defined as the limit of the Hawking masses:
\be
\mathrm{m}_{\mathrm{ADM}}(M^m) =\lim_{s\to \infty} \mathrm{m}_{\mathrm{H}}(s) \in [0,\infty].
\ee
For rotationally symmetric manifolds, this agrees with the definition
of the ADM mass in arbitrary dimensions.

Theorem~\ref{thm-main} concerns manifolds whose ADM mass is
finite and close to $0$
which leads to almost equality in the following 
well known inequality:
\be \label{eqn-Hawking-inequality}
0\le\mathrm{m}_{\mathrm{H}}(0) \le \mathrm{m}_{\mathrm{H}}(s) \le \mathrm{m}_{\mathrm{ADM}}.
\ee
In the next few sections we will see how this constrains
isometric embeddings of the manifolds into Euclidean space
allowing us later to estimate the flat distance between these
spaces and their limits.

\subsection{Riemannian Embedding into $\E^{m+1}$}
\label{ss-Riemannian}

In this section we describe the Riemannian isometric embedding from
our manifold $M^m$ into $\E^{m+1}$ and basic consequences.  Recall
that a Riemannian isometric embedding is a diffeomorphism 
\be\label{eqn-riem-isom-embed}
\psi:M^m\to N^n \textrm{ such that } |\psi_*V|=|V| \,\,\, \forall V\in TM_p.
\ee 
This is not an isometric embedding in the metric sense (see (\ref{eqn-isom-embed})).

\begin{lem}\label{lem-graph}
Given $M^m \in \RS_m$, we can find a rotationally
symmetric Riemannian isometric embedding of $M^m$
into Euclidean space
as the graph of some radial function $z=z(r)$ satisfying $z'(r)\geq 0$.  In graphical coordinates, we have 
\be \label{eqn-graph-g}
g=(1+[z'(r)]^2)dr^2 + r^2 g_0,
\ee
with $r\ge r_{min}$
 and the following formulae for
scalar curvature, 
area, mean curvature, Hawking mass and its derivative
in terms of the radial coordinate $r$:
\begin{alignat}{1}
\mathrm{R} (r) 
&= \frac{m-1}{1+(z')^2}\left(\frac{z'}{r}\right)\left((m-2) \frac{z'}{r}+ \frac{2z''}{1+(z')^2}\right)\\
A (r)& =\omega_{m-1} r^{m-1}\\
H (r)& = \frac{m-1}{r\sqrt{1+(z')^2}}\\
\mathrm{m}_{\mathrm{H}} (r)& = \frac{r^{m-2}}{2}\left(\frac{(z')^2}{1+(z')^2}\right) \\
\mathrm{m}_{\mathrm{H}}' (r)& = \frac{r^{m-1}}{2(m-1)} \mathrm{R}
\end{alignat}
This Riemannian isometric embedding is unique up to a choice of $z_{min}=z(r_{min})$.
\end{lem}

\begin{proof}
First observe that by positivity of the Hawking mass, (\ref{eqn-hawking-1}) and
the lack of interior minimal surfaces (\ref{eqn-f'>0}), we have $f'(s)\in (0,1)$.   
Set $r(s)=f(s)$ and observe that since $s$ is a distance function, 
\be
s'(r)=\sqrt{1+(z'(r))^2}
\ee
which is solvable because $s'(r)\ge 1$.   We choose $z'(r)\ge 0$
which then determines $z(r)$ up to a constant $z_{min}$.
The rest of the equations then follow from the corresponding equations in $s$.
\end{proof}

It is now easy to see the rotational symmetric case of the Positive Mass Theorem
and Penrose Inequality which we restate here as the proof is important to the
almost equality case:

\begin{thm}\label{thm-monotonicity}
Given $M^m \in \RS_m$ isometrically
embedded into Euclidean space as above,
we have
\be
\mathrm{m}_{\mathrm{H}}(r_{min})\le \mathrm{m}_{\mathrm{H}}(r)\le \mathrm{m}_{\mathrm{ADM}}
\ee
and if there is an equality then $M^m$ is Euclidean space (when $\mathrm{m}_{\mathrm{ADM}}=0$) or a Riemannian Schwarzschild manifold of mass $\mathrm{m}_{\mathrm{ADM}}>0$,
\be \label{eqn-def-Sch}
g=\left(1+\frac{2 \mathrm{m}_{\mathrm{ADM}}}{ r^{m-2}-2\mathrm{m}_{\mathrm{ADM}}}\right)dr^2 + r^2 g_0.
\ee
\end{thm}

\begin{proof}
The monotonicity of the Hawking mass follows from 
(\ref{eqn-hawking-increases}).   When there is an equality
we apply Lemma~\ref{lem-graph} to see that
\be
\mathrm{m}_{\mathrm{ADM}} = \frac{r^{m-2}}{2}\left(\frac{(z')^2}{1+(z')^2}\right). 
\ee
So
\be
\left(1+(z')^2\right)2 \mathrm{m}_{\mathrm{ADM}}= r^{m-2} (z')^2
\ee
and
\be
2\mathrm{m}_{\mathrm{ADM}}+  \left( 2 \mathrm{m}_{\mathrm{ADM}} - r^{m-2}\right) (z')^2=0.
\ee
So
\be \label{eqn-Sch}
(z')^2=\frac{-2 \mathrm{m}_{\mathrm{ADM}}}{( 2 \mathrm{m}_{\mathrm{ADM}} - r^{m-2})}.
\ee
When $\mathrm{m}_{\mathrm{ADM}}=0$, $z'(r)=0$ and $z=z_{min}$ is the Euclidean hyperplane.

Observe that $r_{min}$ must then be $0$ because $r_{min}>0$
forces the existence of a minimal surface at the boundary, and
$\partial B_0(r_{min})$ is not minimal in a hyperplane.
\end{proof}

The following lemma will be useful when examining the
deep apparent horizons depicted in Figure~\ref{fig-thm-pted}
that may occur in sequences satisfying the hypothesis of
Theorem~\ref{thm-main}:

\begin{lem}\label{lem-radial}
When $r_{min}>0$,
we can replace the radial coordinate, $r$ by the height coordinate, $z$, so that 
\be
g=(1+[r'(z)]^2) dz^2 + r(z)^2 g_0.
\ee  
Then for $r\ge r_{disk}$ we have the following formulae for scalar curvature,
area, mean curvature, Hawking mass and slope of the Hawking mass
of a level in terms of the height coordinate $z$:
\begin{equation*}
\mathrm{R}(z) = \frac{m-1}{r(1+(r')^2)}\left( \frac{m-2}{r}- \frac{2r''}{1+(r')^2}\right)
=\frac{(m-1)((m-2)(1+(r')^2) -\, 2rr'')}{r^2(1+(r')^2)^2}
\end{equation*}
\begin{alignat}{1}
\mathrm{A}(z) & =\omega_{m-1} r^{m-1}\\
\mathrm{H}(z) & = \frac{(m-1)r'}{r\sqrt{1+(r')^2}}\\
\mathrm{m}_{\mathrm{H}}(z) & = \frac{r^{m-2}}{2(1+(r')^2)}\\
\mathrm{m}_{\mathrm{H}}'(z) & = \frac{r^{m-1} r'}{2(m-1)} \mathrm{R}.
\end{alignat}
When $r_{min}=0$, these formulas hold
outside of a possibly Euclidean disk, $r^{-1}[r_{min}, r_{disk}]$
where $r_{disk} \in [r_{min}, \infty]$,
When $r_{disk}=\infty$ we have Euclidean space.
\end{lem}

\begin{proof}
By Lemma~\ref{lem-graph}, $z'(r)\ge 0$.  Let 
\be
r_{disk}=\sup\{r: \, z'(r) =0\}\in [r_{min},\infty].
\ee
Then all the equations hold for $r>r_{disk}$.

By Lemma~\ref{lem-graph} we have $\mathrm{m}_{\mathrm{H}}(r_{disk})=0$, so
by the Positive Mass Theorem, $\mathrm{m}_{\mathrm{H}}(r_{min})=0$,
so $r_{min}=0$ and $r^{-1}[0,r_{disk})$ is a ball.  It is clearly
a Euclidean disk by (\ref{eqn-graph-g}).  
\end{proof}

\subsection{Bounding $\diam(\partial M)$ and $F'$}
\label{ss-bounding}

In this subsection we use the ADM mass to provide Lipschitz control on $z=F(r)$ on annular regions
[Lemma~\ref{lem-F'}] and to bound the diameter of the
boundary of the manifold [Lemma~\ref{lem-rmin}].
These lemmas will be applied later to prove our stability theorems 
[Theorem~\ref{thm-main}].

\begin{lem} \label{lem-rmin}
If $M^m\in \RS$ then
\be
r_{min}\le \left(2\mathrm{m}_{\mathrm{ADM}}\right)^{1/(m-2)}.
\ee
So $\diam(\partial M^m) \le \pi \left(2\mathrm{m}_{\mathrm{ADM}}\right)^{1/(m-2)}.$
\end{lem}

\begin{proof}
Assuming $r_{min}>0$, we know by Lemma~\ref{lem-radial}
with $z_{min}=z(r_{min})$ that
\be
0 = \frac{(m-1)r'(z_{min})}{r_{min}\sqrt{1+(r'(z_{min}))^2}}\\
\ee
because the boundary is a minimal surface.  So
$r'(z_{min})=0$ and the Hawking mass is
\be
\mathrm{m}_{\mathrm{H}}(z_{min})  = \frac{r_{min}^{m-2}}{2(1+0^2)}.
\ee
The lemma then follows from the monotonity of Hawking mass
in (\ref{eqn-hawking-increases}).
\end{proof}

\begin{lem}\label{lem-F'}
Using the graphical coordinates of Lemma~\ref{lem-graph}, with $z=F(r)$, we
have
\be
F'(r) \ge \sqrt{\frac{2m_1}{r^{m-2}-2m_1 }\,} \qquad \forall r \in [r_1, \infty)
\ee
for any $r_1\ge r_{min}$ where $m_1=\mathrm{m}_{\mathrm{H}}(r_1)$ and
\be
F'(r) \le \sqrt{\frac{2\mathrm{m}_{\mathrm{ADM}}}{r^{m-2}-2\mathrm{m}_{\mathrm{ADM}}}\, } 
\qquad \forall r \ge \max \left\{ r_1, \left(2\mathrm{m}_{\mathrm{ADM}} 
\right)^{1/(m-2)}\right\}
\ee
where $\mathrm{m}_{\mathrm{ADM}}=\mathrm{m}_{\mathrm{ADM}}(M^m)$.
\end{lem}

\begin{proof}
By the formulas in Lemma~\ref{lem-graph} and the monotonicity of Hawking mass
in (\ref{eqn-hawking-increases})
we have the following for $r>r_1$:
\begin{alignat}{2}
 \mathrm{m}_{\mathrm{H}}(r_1)&\leq \mathrm{m}_{\mathrm{H}}(r) \leq  \mathrm{m}_{\mathrm{ADM}}\\
  \mathrm{m}_{\mathrm{H}}(r_1)&\le  \frac{r^{m-2}}{2}\left(\frac{(z')^2}{1+(z')^2} \right) \leq  \mathrm{m}_{\mathrm{ADM}}\\
 2\mathrm{m}_{\mathrm{H}}(r_1) (1+(z')^2)&\leq r^{m-2}(z')^2 \leq 2m_{\mathrm{ADM}}(1+(z')^2 )
\end{alignat}
So we get
\begin{alignat}{1}
 2\mathrm{m}_{\mathrm{H}}(r_1)&\leq (r^{m-2}-2\mathrm{m}_{\mathrm{H}}(r_1))(z')^2\\
 2\mathrm{m}_{\mathrm{ADM}}&\geq (r^{m-2}-2\mathrm{m}_{\mathrm{ADM}})(z')^2
 \end{alignat}
The first equation tells us that 
\be
z' \geq \sqrt{\frac{2m_1}{r^{m-2}-2m_1}} \qquad \forall r>r_1.
\ee
The second implies that 
\be
z' \leq \sqrt{\frac{2\mathrm{m}_{\mathrm{ADM}}}{r^{m-2}-2\mathrm{m}_{\mathrm{ADM}}}} \qquad \forall r\ge (2\mathrm{m}_{\mathrm{ADM}})^{m-2}.
\ee
\end{proof}


\subsection{Rotationally Symmetric Examples}
\label{ss-ex-1}

Here we review the key examples depicted in Figure~\ref{fig-thm-pted}
which inspired the use of the Intrinsic Flat Distance to estimate the
stability of the Positive Mass Theorem.   These are all well known
examples but we present them for completeness of exposition.

Recall (\ref{eqn-hawking-increases}) implies that, 
in the rotationally symmetric setting, monotonicity of the Hawking mass 
on the symmetric spheres is \emph{equivalent} to nonnegativity of scalar curvature.  Therefore, we have the following lemma which will be useful
for constructing examples:

\begin{lem}\label{lem-ex}
There is a bijection between elements of
$\RS_m$ and increasing functions
$\mathrm{m}_{\mathrm{H}}:[\rmin,\infty)\to\R$ such that 
\be \label{lem-ex-1}
\mathrm{m}_{\mathrm{H}}(\rmin)=\frac{1}{2}\rmin^{m-2}
\ee 
and
\be
\mathrm{m}_{\mathrm{H}}(r)<\frac{1}{2}r^{m-2}
\ee 
for $r>\rmin\ge 0$.  In this section we will call these functions
\emph{admissible Hawking mass functions}.  
\end{lem}

\begin{proof}
Given $M^m \in \RS$, apply Lemma~\ref{lem-graph}, to
determine $r_{min} \ge 0$.  
Since $M^m$ has no closed interior minimal surfaces, $z'(r)>0$
so
\be
\mathrm{m}_{\mathrm{H}}(r)=\frac{r^{m-2}}{2}\left(\frac{(z'(r))^2}{1+(z'(r))^2}\right)<\frac{1}{2}r^{m-2}.
\ee
If $r_{min}=0$, then $\mathrm{m}_{\mathrm{H}}(r_{min})=0$.  If $r_{min}>0$, we
have $\lim_{r\to r_{min}} z'(r)=\infty$, so we have (\ref{lem-ex-1}).

Given an admissible Hawking function, $\mathrm{m}_{\mathrm{H}}:[r_{min}, \infty) \to \R$,
we define $z: [0, \infty) \to \R$,
via the formula
\be \label{lem-ex-m-H-to-z}
 z(\bar{r}) = \int_{\rmin}^{\bar{r}} \sqrt{ \frac{2\mathrm{m}_{\mathrm{H}}(r)}
{r^{m-2}-2\mathrm{m}_{\mathrm{H}}(r)}}\,dr.
\ee
This determines a rotationally symmetric manifold.   Since
\be \label{lem-ex-z'}
z'(r) =  \sqrt{ \frac{2\mathrm{m}_{\mathrm{H}}(r)}{r^{m-2}-2\mathrm{m}_{\mathrm{H}}(r)}}>0 \qquad \forall r> r_{min}
\ee
we have no interior  minimal surfaces.  
If $r_{min}=0$
then $z'(r_{min})=0$ and if $r_{min}>0$ then
\be
\lim_{r\to r_{min}}z'(r)\ge\lim_{r\to r_{min}}
\sqrt{ \frac{2\mathrm{m}_{\mathrm{H}}(r)}{r^{m-2}-2\mathrm{m}_{\mathrm{H}}(r)}}=\infty
\ee
so the boundary is an outermost minimal surface. 
\end{proof}

We begin with the most
basic example depicted in column two of Figure 1: Schwarzschild manifolds
whose ADM mass converges to $0$.  

\begin{example}\label{ex-Sch-to-0}  
The Riemannian Schwarzschild space, $M^m_{\mathrm{Sch}}$ of
mass $\mathrm{m}_{\mathrm{ADM}}$ can be found by applying Lemma~\ref{lem-ex}
with $\mathrm{m}_{\mathrm{H}}(r)$ constant equal to $\mathrm{m}_{\mathrm{ADM}}$.   Its metric
satisfies (\ref{eqn-def-Sch}).  These spaces are diffeomorphic
to Euclidean space $\E^m$ with a ball of radius $r_{min}$ removed.
Fixing an area $\alpha_0>0$, we see that outside a rotationally
symmetric sphere $\Sigma_{\alpha_0}$ of area $\vol_{m-1}(\Sigma_{\alpha_0})=\alpha_0$
the metric converges smoothly to the Euclidean metric.
However, these manifolds are not diffeomorphic to Euclidean
space and we do not have smooth convergence globally.
\end{example}

Next we consider the deep gravity wells depicted in third and fourth
columns of Figure~\ref{fig-thm-pted}.  First we provide a general
lemma describing which admissable Hawking masses lead to
strongly vertical graphs $z=F(r)$:

\begin{lem} \label{lem-steep}
 Let $\epsilon>0$.  We
choose an admissible Hawking mass function 
$\mathrm{m}_{\mathrm{H}}:[r_{min},\infty)\too\rr$ such that
\be
\mathrm{m}_{\mathrm{H}}(r)\ge\frac{1}{2}r^{m-2}\left(1-\epsilon^2\right)
\ee
 on the interval $[r_1,r_2]$.
Then the distance from the level $r^{-1}(r_1)$ to the
level $r^{-1}(r_2)$ in the corresponding manifold
is greater than 
\be
z(r_2)-z(r_1)\ge (r_2-r_1)\sqrt{\frac{(1-\epsilon^2)}{\epsilon^2}}.
\ee
\end{lem}

\begin{proof}
By (\ref{lem-ex-z'}) we have
\be
z'(r) \ge  \sqrt{ \frac{r^{m-2}(1-\epsilon^2)}{r^{m-2}-r^{m-2}(1-\epsilon^2)}}
=\sqrt{\frac{(1-\epsilon^2)}{\epsilon^2}}.
\ee
\end{proof}

We now apply this to state and prove the example of the deep horizon
depicted in the third column of Figure~\ref{fig-thm-pted}:

\begin{example}\label{ex-deep-well}
Given $L>0$ and $\alpha_0>0$ and $\delta>0$
we claim we can construct $M^m\in \RS_m$
with $\mathrm{m}_{\mathrm{ADM}}(M^m)< \delta$ such that the
distance $d(\Sigma_{min}, \Sigma_{\alpha_0}) >L$ where $\Sigma_{\alpha_0}$
is a level set of area $\vol_{m-1}(\Sigma_{\alpha_0})=\alpha_0$
and $\Sigma_{min}$ is either the boundary of $M^m$
or the pole.
\end{example}

In Section~\ref{Sect-Gromov} we use this example to find
a sequence of Riemannian manifolds whose ADM mass approaches
$0$ that does not converge in the Gromov-Hausdorff sense
to Euclidean space due to the thin deep wells.
[Example~\ref{ex-not-GH}].

\begin{proof}
Let $r_0$ be defined so that $\omega_{m-1}r_0^{m-1}=\alpha_0$.
Let 
\be \label{ex-eq-6}
\mathrm{m}_{\mathrm{H}}(r)=\delta'=\min\left\{\delta/2, r_0^{m-2}/2\right\}   \textrm{ for } r\ge r_A.
\ee
Given any $\epsilon>0$, choose $r_\epsilon\in (0,r_0)$ so that
\be\label{ex-eq-7}
\frac{1}{2}r_\epsilon^{m-2}(1-\epsilon^2)=\delta'.
\ee
If $\mathrm{m}_{\mathrm{H}}(r)$ is a smooth function such that
\be \label{ex-eq-5}
\frac{1}{2}r^{m-2}(1-\epsilon^2)\le\mathrm{m}_{\mathrm{H}}(r)\le\frac{1}{2}r^{m-2}
\textrm{ for } r \in [r_\epsilon/2, r_\epsilon]
\ee
then
\begin{eqnarray}
z(r_\epsilon)-z(r_\epsilon/2)
&\ge& \frac{r_\epsilon\sqrt{1-\epsilon^2}}{2\epsilon}
\qquad \qquad \textrm{ by Lemma~\ref{lem-steep}, }\\
&=&
\frac{\sqrt{1-\epsilon^2}}{2\epsilon} 
\left(\frac{2\delta'}{(1-\epsilon^2)}\right)^{1/(m-2)}
\textrm{ by (\ref{ex-eq-7}),}\\
&\ge& L \qquad \textrm{ for } \epsilon \textrm{ sufficiently small
and fixed }\delta'.
\end{eqnarray}
Finally choose $r_{min}=r_\epsilon/4$ or $0$ and choose a smooth
admissible Hawking function satisfying (\ref{ex-eq-6})
and apply Lemma~\ref{lem-ex}.
\end{proof}

\section{Isometric Embeddings and the Intrinsic Flat Distance}
\label{Sect-Embed}

In this section we provide techniques for constructing
explicit filling manifolds to estimate the Intrinsic
Flat Distance.

\subsection{Review}\label{ss-review}

In the definition of the Intrinsic Flat Distance, one uses metric isometric embeddings
(a la Gromov):
\be \label{eqn-isom-embed}
\varphi: X\to Z \textrm{ such that }
d_Z(\varphi(x_1), \varphi(x_2))= d_X(x_1,x_2) \qquad \forall x_1, x_2\in X.
\ee
In contrast, one often finds Riemannian isometric embeddings (a la Nash)
as defined in (\ref{eqn-riem-isom-embed}).
Riemannian isometric embeddings are not necessarily 
metric isometric embeddings.

The metric space property of a Riemannian isometric embedding is a length space
property.   Recall that a length space is a metric space $(X, d)$ such that
\be
d_X(x_1, x_2) =\inf \{ L_X(C): \,  C(0)=x_1, \, C(1)=x_2\}
\ee
where the length, $L_X(C)$, of the curve $C: [0,1]\to X$ is the rectifiable length using $d_X$.
Given a rectifiably connected subset $Y \subset Z$, it has an induced metric
\be
d_Y(y_1, y_2) =\inf \{ L_Y(C): \,  C(0)=y_1, \, C(1)=y_2\}
\ee
where the length of the curve $C: [0,1]\to Y \subset Z$ is the rectifiable length 
using $d_Z$.  Observe that $(Y, d_Y)$ (the induced metric) is then a length space while
$(Y, d_Z)$ (the restricted metric) is just a metric space and
\be
d_Y(y_1, y_2) \ge d_Z(y_1, y_2).
\ee  
Consider as an example
\be
Y=\{(x,y,z): \, x^2+y^2+z^2=1\} \subset \E^3.
\ee
Here the restricted metric is the distance measured using line segments 
while the induced metric or intrinsic metric is the distance measured in the sphere.

Riemannian manifolds are length spaces.  If $\varphi: M \to N$ is a Riemannian
isometric embedding, then $L_M(C)= L_N(\varphi\circ C)$ for all curves $C:[0,1]\to M$.
Thus $\varphi$ is an isometric embedding from $M$ to its image, $\varphi(M)$,
where the image is endowed with the induced metric:
\be
d_M(p_1,p_2)=d_{\varphi(M)}(\varphi(p_1), \varphi(p_2)) \ge 
d_N(\varphi(p_1), \varphi(p_2))
\qquad \forall p_1, p_2 \in M.
\ee
In fact it is an isometry onto its image with the induced length metric.
However it is not an isometric embedding into $N$ unless the image $\varphi(M)$
is convex in $N$.  When $\varphi(M)$ is convex, the infimums are achieved by length
minimizing curves that lie within the set, and so, in that case, it is an isometric
embedding.  A plane is convex in $\E^3$, so it is isometrically embedded.
The equatorial sphere in a 3-sphere is isometrically embedded
into the three sphere.

Many examples of isometric embeddings are given in \cite{Gromov-filling} and 
in \cite{SorWen2} as they are an essential ingredient towards the explicit
computation of filling volumes and Intrinsic Flat Distances.  Among these is the
classic warped product:  

\begin{thm}\label{thm-warp}
Given a warped product manifold, $M^m=\R \times_f S^{m-1}$ with metric
$g_M=dr^2 + f(r)^2 g_{0}$ where $g_0$ is the standard metric on $S^{m-1}$.
This isometrically embeds into $N^{m+1}=\R \times_f S^{m}$ with metric
$g_M=dr^2 + f(r)^2 g_{0}$ where $g_0$ is the standard metric on $S^{m}$
via an isometric embedding which preserves the radial coordinate, $r$,
and maps each sphere into the equatorial sphere of that level.
\end{thm}

Rotationally symmetric subsets of Euclidean space do not isometrically
embed into Euclidean space.  However, they can be viewed as warped 
products and be isometrically embedded into rotationally symmetric
submanifolds of higher dimensional Euclidean space:
\begin{equation*}
\varphi:M=\left\{(x,y,z):\, z=f(x^2+y^2)\right\}\to 
N=\left\{(x,y,z,w):\, z=f(x^2+y^2+w^2)\right\}
\end{equation*}
where $M$ and $N$ are endowed with the induced length metrics
and $\varphi(x,y,z)=(x,y,z,0)$.

We close this review section with the following theorem which could be
useful in applications of Intrinsic Flat Distance to general relativity.  The 
isometric embeddings given in Nirenberg's theorem 
applied in the work of Shi-Tam satisfy the hypothesis 
of this theorem \cite{Nirenberg-surface} \cite{Shi-Tam}.

\begin{thm}\label{thm-convex}
If $\varphi: M^m \to \E^{m+1}$ is a Riemannian isometric embedding such
that $\varphi(M^m)=\partial K$ where $K$ is a closed convex set in $E^{m+1}$,
then $\varphi: M^m \to Cl(\,\E^{m+1}\setminus K)$ is an isometric embedding.
\end{thm}

While this theorem must be classical, its proof isn't readily available for
citation, so we include it here:

\begin{proof}
Let $p_1, p_2\in M^m$ be joined by $C:[0,1] \to Cl(\E^{m+1}\setminus K)$,
which is the shortest among all such curves.  We know the shortest exists
by applying the Arzela Ascoli Theorem keeping in mind that a sequence of curves
of decreasing length remains in a compact subset of $\E^{m+1}$.

If the image of $C$ lies in $\partial K$, then we are done.  Assume on the
contrary, that it does not.   Let
\be
t_1 = \sup\bigg\{ t: \, C([0,t])\subset \partial K \bigg\} \in [0,1)
\ee
and let 
\be
t_2 = \sup\left\{ t: \, C((t_1,t))\subset \E^{m+1}\setminus K \right\} \in (t_1,1].
\ee
Since $C$ is a length minimizing curve, its restriction to $[t_1, t_2]$
is minimizing from $C(t_1)$ to $C(t_2)$.  This segment lies in the open
set $\E^{m+1}\setminus K$, so variation of arclength within this
flat region proves it is a straight Euclidean line segment.
Since $\varphi(C(t_1)), \varphi(C(t_2))\subset K \subset \E^{m+1}$ and $K$ is convex,
this segment must lie in $K$.  This is a contradiction.
\end{proof}

\subsection{Constructing Isometric Embeddings}  
\label{ss-constr}

In this subsection we prove the following theorem which will 
later be applied to estimate the Intrinsic Flat Distances between spaces.
This theorem is depicted in Figure 2.

\begin{thm}\label{thm-embed-const}
Let $\varphi: M \to N$ be a Riemannian isometric embedding and let
\be \label{eqn-embed-const-1}
C_M:= \sup_{p,q\in M} \left( d_M(p, q) - d_N(\varphi(p),\varphi(q)) \right).
\ee
If
\be
Z=\left\{(x,0): \,x\in N\} \cup \{(x,s): \, x\in \varphi(M),\, s\in [0, S_M]\right\} 
\subset N \times [0, S_M]
\ee
where
\be\label{eqn-embed-const-S}
S_M= \sqrt{C_M(\diam(M)+C_M)}
\ee
then $\psi: M \to Z$ defined as $\psi(x)=(\varphi(x), S_M)$, is an 
metric isometric
embedding into $(Z, d_Z)$ where $d_Z$ is the induced length metric 
from the
isometric product metric on $N \times [0,S_M]$.
\end{thm}

\begin{figure}[h] 
   \centering
   \includegraphics[width=3in]{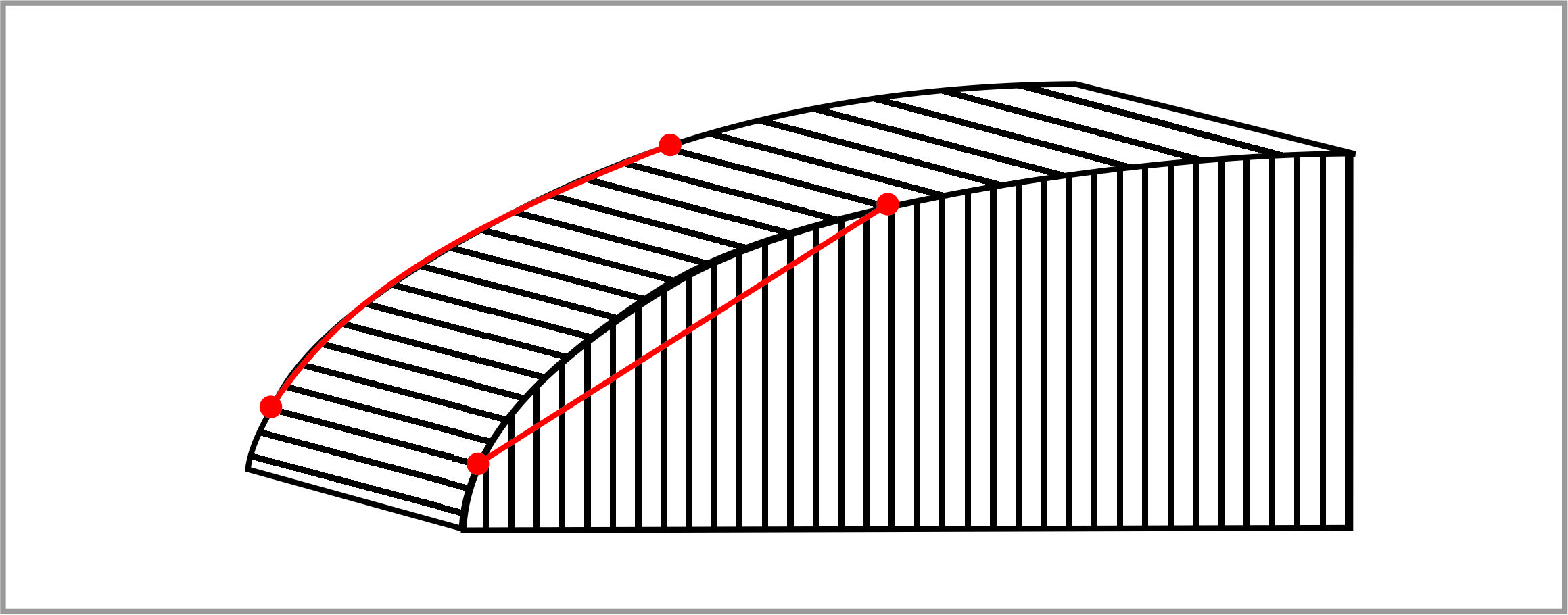} 
   \caption{Explicit Isometric Embedding into Z
   }
   \label{fig-iso-embed}
\end{figure}

Later we will provide techniques for estimating the value of 
the ``embedding constant" $C_M$.


\begin{proof}
First observe that $\psi: M \to N \times [0,S_M]$ is a Riemannian isometric
embedding.

Let $p, q\subset M$.   Let $C_i: [0,1] \to Z$ be curves parametrized
proportional to arclength running
from $\psi(p)$ to $\psi(q)$ such that 
\be
\lim_{i\to\infty} L(C_i) = d_Z(\psi(p), \psi(q)).
\ee
Observe that a closed ball in $Z$ is compact, and the $C_i$ are equicontinuous,
so by the Arzela-Ascoli Theorem, a subsequence of the $C_i$ converge to a length minimizing
curve, $C:[0,1]\to Z$, parametrized proportional to arclength 
such that $C(0)=p$ and $C(1)=q$ and
$L(C)=d_Z(\psi(p), \psi(q))$.   

If the image of $C$ lies in $\psi(M)$, then $C$ is the shortest curve in $\psi(M)\subset Z$
from $\psi(p)$ to $\psi(q)$.   Since $\psi: M \to N\times [0, S_M]$ is a 
Riemannian isometric embedding, $\psi: M \to \psi(M)$ is an isometry.
Thus there is a curve $\gamma: [0,1]\to M$ running from $p$ to $q$ such that
$\psi\circ \gamma=C$.  Furthermore $\gamma$ is length minimizing in $M$ and
parametrized proportional to arclength, so it is a minimizing geodesic in $M$.
So
\be
d_M(p,q)= L(\gamma)=L(C) =d_Z(\psi(p), \psi(q)).
\ee
Thus we need only show the image of $C$ lies in $\psi(M)$.

We will write $C(t)=(x(t),s(t))$.  Let 
\be
T_M=\bigg\{t\in [0,1]: \, C(t) \subset \varphi(M)\times[0,S_M]\bigg\}
\ee
and let
\be
T_N=\bigg\{t\in [0,1]: \, C(t) \subset N\times\{0\}\bigg\}.
\ee
If $[t_1, t_2]\subset T_M$ then (since $C$ is length minimizing
in an isometric product) the length of this interval of $C$ satisfies
\be \label{eqn-C-in-T_M}
L(C[t_1,t_2])= \sqrt{ d_{\varphi(M)}(x(t_2),x(t_1))^2  + (s(t_2)-s(t_1))^2 \,}
\ee
with $x[t_1,t_2]$ the image of a minimizing geodesic segment
in $M$ and $s[t_1,t_2]=[s(t_1),s(t_2)]$.
If $[t_1,t_2]\subset T_M$ and $t_i \subset T_N$, then $s(t_i)=0$
so in fact this segment of $C$ lies 
in $\varphi(M)\times \{0\}$.     Thus
\be 
T_M=[0, m_1] \cup [m_2,  1].
\ee
If $m_1=m_2$, then we can apply (\ref{eqn-C-in-T_M}) with $t_0=0$
and $t_1=1$ and the fact that $s(0)=s(1)=S_M$, to see that
the image of $C$  lies in $\varphi(M) \times \{S_M\}=\Psi(M)$ and we are done.

Assume on the contrary that $m_2>m_1$.
Observe that since the image of $C$ lies in $Z$,
$C:[m_1, m_2] \to  N \times \{0\}\subset Z$.  Furthermore
\be
C(m_1), C(m_2) \in \varphi(M)\times \{0\}.
\ee 
Since $C$ is length minimizing in $Z$, it is length minimizing between
$C(m_1)=(x(m_1),0)$ and $C(m_2)=(x(m_2),0)$.  Thus
\begin{eqnarray}
L(C[m_1, m_2])&=& d_{N\times\{0\}}(C(m_1), C(m_2))\\
&=& d_N(x(m_1), x(m_2))\\
&\ge& d_{\varphi(M)}(x(m_1), x(m_2)) - C_M. \label{eqn-mid-segment}
\end{eqnarray}
We will next sum up the lengths of the three segments of $C$
to reach a contradiction.  As this will involve the length on
the isometric product region, we first observe some properties
on these regions.
Let
\be
X=X(a,b)=d_{\varphi(M)}(x(a), x(b)) \le \diam(M). 
\ee
By our choice of $S_M$
we have
\be
S_M^2 =C_M\left(\diam(M)+ C_M\right) \ge  X C_M + C_M^2. 
\ee
and so
\be
X^2+ S_M^2 > X^2 + 2X C_M/2 + C_M^2/4.
\ee
By (\ref{eqn-C-in-T_M}), if $[a,b]\subset T_M$ and
$|s(a)-s(b)|= S_M$ then
\be
L(C[a,b])= \sqrt{ X(a,b)^2 + S_M^2\,}\, >\,X(a,b)+ C_M/2.
\ee

Combining this with (\ref{eqn-mid-segment}) and the fact
that $s(0)=s(1)=S_M$ and $s(m_1)=s(m_2)=0$, we have
\be
L(C)= L(C[0,m_1]) +  L(C[m_1,m_2]) + L(C[m_2,1]) 
\ee
where
\begin{eqnarray}
L(C[0,m_1]) &=&\sqrt{X(0,m_1)^2+ (s(0)-s(m_1))^2\,} \\
  &=&\sqrt{X(0,m_1)^2 + S_M^2\,}\\
 &>&X(0,m_1) + C_M/2 \\
  &=&d_{\varphi(M)}(x(0), x(m_1)) + C_M/2
   \end{eqnarray}
 \begin{eqnarray}
L(C[m_1,m_2])  &\ge & X(m_1,m_2) - C_M \\
  &=&d_{\varphi(M)}(x(m_1), x(m_2)) -C_M
 \end{eqnarray}
 \begin{eqnarray}
L(C[m_2,1]) 
 &=& \sqrt{X(m_2,1)^2 + (s(m_2)-s(1))^2\,}\\
 &=&  \sqrt{X(m_2,1)^2 + S_M^2\,}\\
 &>& X(m_1, 1) + C_M/2 \\
 &=&d_{\varphi(M)}(x(m_2), x(1)) + C_M/2.
 \end{eqnarray}
Thus by the Triangle Inequality we have
\be
L(C)> d_{\varphi(M)}(x(0),x(1))=d_{\varphi(M)}(p,q).
\ee
This contradicts $L(C) =d_Z(\psi(p), \psi(q)) \le d_{\varphi(M)}(p,q)$.
\end{proof}

\subsection{Estimating the Intrinsic Flat Distance}
\label{ss-est}

In this subsection we prove two general propositions that can
be applied to bound the Intrinsic Flat Distance between Riemannian
manifolds that have Riemannian isometric embeddings into a
common Riemannian manifold.  Recall the bound on the
Intrinsic Flat Distance given in the introduction in (\ref{eqn-def-intrinsic-flat-1})
and (\ref{eqn-def-intrinsic-flat-2}) require a metric isometric embedding
so we apply Theorem~\ref{thm-embed-const}.  The first proposition
is clear and easy to see while the second is a bit more complicated but
necessary to prove Theorem~\ref{thm-main}.

\begin{prop}\label{embed-const}
If $\varphi_i: M^m_i \to N^{m+1}$ are Riemannian isometric embeddings with 
embedding constants $C_{M_i}$ as in (\ref{eqn-embed-const-1}), and if
they are disjoint and lie in the boundary of a region $W \subset N$
then
\begin{eqnarray}
d_{\mathcal{F}}(M_1, M_2) &\le& 
S_{M_1}\left(\vol_m(M_1)+ \vol_{m-1}(\partial M_1) \right) \\
&&+S_{M_2}\left(\vol_m(M_2)+ \vol_{m-1}(\partial M_2) \right)\\
&&+ \vol_{m+1}(W) + \vol_{m}(V)
\end{eqnarray}
where $V= \partial W \setminus \left( \varphi_1(M_1) \cup \varphi_2(M_2)\right)$
where $S_{M_i}$ are defined in (\ref{eqn-embed-const-S}).

\end{prop}

\begin{proof}
We first create a piecewise smooth manifold, 
\be
Z^{m+1}
=(M_1\times [0,S_{M_1}] )\,\,\cup\,\, (M_2\times [0,S_{M_2}])
\,\, \cup\,\, W^{m+1}
\ee
where the regions are glued together along the Riemannian
isometric embeddings $\varphi_i(M_i)\subset W^{m+1}$
to $M_i \times \{0\}$ to form $Z$.  
Applying Theorem~\ref{thm-embed-const},  
we have metric isometric embeddings 
$\psi_i: M_i \to Z$ defined by $\psi_i(x)=(\varphi_i(x), S_{M_i})$. 
Setting our filling manifold $B^{m+1}=Z^{m+1}$ as in
(\ref{eqn-def-intrinsic-flat-2}), we then have an excess boundary
\be
A^m=(\partial M_1)\times [0,S_{M_1}] 
\,\,\cup \,\,(\partial M_2)\times [0,S_{M_2}]
\,\,\cup\,\, V.
\ee
The proposition then follows from (\ref{eqn-def-intrinsic-flat-1}).
\end{proof}

The next proposition will be applied to prove Theorem~\ref{thm-main}.
It concerns pairs of manifolds which do not have global
Riemannian isometric embeddings into a common manifold $U^{m+1}$:

\begin{prop}\label{embed-const-2}
If $M^m_i$ are Riemannian manifolds and $U^m_i\subset M^m_i$
are submanifolds that have
Riemannian isometric embeddings $\varphi_i: U^m_i \to N^{m+1}$ with 
embedding constants $C_{U_i}$ as in (\ref{eqn-embed-const-1}), and if
their images are disjoint and lie in the boundary of a region $B_1 \subset N$
then
\begin{eqnarray}
d_{\mathcal{F}}(M_1, M_2) &\le& 
S_{U_1}\left(\vol_m(U_1)+ \vol_{m-1}(\partial U_1) \right)\\
&&+S_{U_2}\left(\vol_m(U_2)+ \vol_{m-1}(\partial U_2) \right)\\
&&+ \vol_{m+1}(B_1) + \vol_{m}(V) \\
&&+\vol_{m}(M_1\setminus U_1)
+\vol_m(M_2\setminus U_2)
\end{eqnarray}
where 
$V= \partial B_1 \setminus \left( \varphi_1(U_1) \cup \varphi_2(U_2)\right)$
where $S_U$ are defined in (\ref{eqn-embed-const-S}).
\end{prop}

\begin{proof}
Let $S_i=S_{U_i}$ as in Theorem~\ref{thm-embed-const}.
We first create a piecewise smooth manifold, 
\begin{eqnarray}
Z^{m+1}
&=&(U_1\times [0,2S_1] )\,\,\cup\,\, (U_2\times [0,2S_2])
\,\, \cup\,\, B_1^{m+1} \\
&& \cup \,\,
(M_1\setminus U_1)\times [S_1, 2S_1] \,\,\cup\,\, (M_2\setminus U_2)\times [S_2,2S_2]
\end{eqnarray}
where the regions are glued together along the Riemannian
isometric embeddings $\varphi_i(U_i)\subset B_1^{m+1}$
to $U_i \times \{0\}$ 
and along $\partial U_i \times [S_i,2S_i]$
to form $Z$.  
Applying Theorem~\ref{thm-embed-const},  
we have metric isometric embeddings 
$\psi_i: U_i \to Z$ defined by $\psi_i(x)=(\varphi_i(x), S_i)$. 
In fact these extend to isometric embeddings
$\psi_i: M_i \to Z$ defined by $\psi_i(x)=(\varphi_i(x), S_i)$
since this is a metric isometric embedding 
on $M_i\setminus U_i$ and any path in $Z$ running
from $M_i \setminus U_i \times [S_i,2S_i]$
to $U_i\times [0,2S_i]$ must pass through
$\partial U_i \times [S_i,2S_i]$ and would be shorter
if it stayed in $M_i\times \{S_i\}$. 

Our filling manifold is chosen to be
\be
B^{m+1}
=(U_1\times [0,S_1] )\,\,\cup\,\, (U_2\times [0,S_2])
\,\, \cup\,\, B_1^{m+1}.
\ee
Then by
(\ref{eqn-def-intrinsic-flat-2}), we then have an excess boundary
\begin{eqnarray}
A^m&=&(\partial U_1)\times [0,S_1] 
\,\,\cup \,\,(\partial U_2)\times [0,S_2]
\,\,\cup\,\, V \\
&&\cup \,\, M_1\setminus U_1 \,\, \cup \,\, M_2 \setminus U_2.
\end{eqnarray}
The proposition then follows from (\ref{eqn-def-intrinsic-flat-1}).
\end{proof}

\subsection{The Embedding Constant for Graphs}
\label{ss-embed}

In this section we provide means for estimating the embedding constant, $C_M$,
as defined in (\ref{eqn-embed-const-1}):
\be 
C_{M}:= \sup_{p,q\in M} \,\,\, \bigg(d_M(p,q) - d_N(\varphi(q),\varphi(q)) \bigg).
\ee
when the manifold $M$ has a Riemannian isometric embedding $\varphi:M^m \to \E^{n}$ defined as a graph.

\begin{thm}\label{thm-Z}
Let $M^m$ be a compact Riemannian manifold with boundary
defined by the graph
\be
M^m=\{(x,z):\, z=F(x), \, x\in W\} \subset W \times \R
\ee
where $F: W\to \R$ is differentiable and $W$ is
a Riemannian manifold with boundary.
Viewed
as a Riemannian isometric embedding into $W\times \R$
the embedding constant satisfies
\be \label{eqn-C_M-bound-1}
C_M\le 2\diam(W) \sup\{|\grad F_x|: \, x\in W\}. 
\ee
\end{thm}

Later we will apply this theorem with 
\be
W=Ann_0(R_0,R_1) \subset \E^m
\ee
for rotationally symmetric $F$.  

\begin{rmrk}If one examines the
proof one can see that $C_M$ is really bounded by
an integral of $|\grad F|$ over a length minimizing
curve in $W$.  However, the estimate we've written
in Theorem~\ref{thm-Z} suffices for our purposes.
\end{rmrk}

\begin{proof}
Since $M^m$ is compact, there exists a pair of points $p_0,p_1\in M$ 
such that
\be \label{eqn-est-C_M-1}
C_{M}=  d_{M^m}(p_0, p_1) - d_{W \times \R}(p_0, p_1).
\ee
We write $p_i=(x_i, z_i)$.  

Let 
$C$ be a length minimizing curve
in $W\times \R$ from $p_0$ to $p_1$.    
We write $C(t)=(x(t), z(t))\in W \times \R$.  
Then $x(t)$ is a length minimizing curve
in $W$ from $x_0$ to $x_1$ because it is the projection (in
an isometric product) of a length minimizing curve.
Let 
\be \label{eqn-est-C_M-15}
h:=d_W(x_0,x_1)\le \diam(W).
\ee
We now parametrize $C(t)$ and $x(t)$ so that
$x:[0,h]\to W$ is parametrized by arclength
and $x(0)=x_0$, $x(h)=x_1$, $z(0)=z_0$ and $z(h)=z_1$.

Observe that $\{(x(t),z): t\in [0,h], z\in \R\}$ is isometric
to a flat Euclidean strip $[0,h]\times \R$ with the metric restricted
from $W\times \R$.  The isometry $\psi(t,z)=(x(t),z)$.   This implies that
$z(t)$ is linear in $t$ and
\be \label{eqn-est-C_M-2}
z'(t)= \frac{(z_1-z_0)}{h}
\ee
Note that $x(t)$ is length minimizing in a manifold with boundary so it is
not necessarily a smooth geodesic.  However it is smooth away from a discrete
set of points.   Where it is smooth $g_W(x'(t), x'(t))=1$.  

Define $\tilde{C}(t)=(\tilde{x}(t), \tilde{z}(t)) \in M\subset W\times \R$
where $\tilde{x}(t)=x(t)$ and $\tilde{z}(t)=F(x(t))$.   
Observe that
where $x(t)$ is smooth, we have
\be \label{eqn-est-C_M-7}
|\tilde{z}'(t)|=|\grad F_{x(t)}|.
\ee
Since $p_0,p_1\in M$, $\tilde{z}(0)=z(0)$ and $\tilde{z}(h)=z(h)$.
Thus $\tilde{C}$ is a curve from $p_0$ to $p_1$ in $M$.

Let $\gamma$ be a length minimizing curve in $M$
from $p_0$ to $p_1$.  Then by (\ref{eqn-est-C_M-1})
\be
C_{M}=  L(\gamma)-L(C)\le L(\tilde{C})-L(C).
\ee

Since $x(t)$ is smooth on a set of full measure in $[0,h]$,
so are $C$ and $\tilde{C}$, and  we have
\begin{eqnarray}
C_M&=& \int_0^h g_{W\times \R}(\tilde{C}'(t),\tilde{C}'(t))^{1/2}\,dt
-\int_0^h g_{W\times \R}({C}'(t),{C}'(t))^{1/2}\,dt\\
&=& \int_0^h (1 + (\tilde{z}'(t))^2)^{1/2}- (1 + ({z}'(t))^2)^{1/2}\,dt\\
&=& \int_0^h (1 + (\tilde{z}'(t))^2)^{1/2}- (1 + (z_1-z_0)^2/h^2)^{1/2}\,dt.
\end{eqnarray}
Let 
\be
T=\{ t\in [0,h] :  |\tilde{z}'(t)| \ge |z_1-z_0|/h\}.
\ee
Since $\tilde{z}$ is smooth away from a finite collection of points,
we see that there exists 
\be
0\le a_1<b_1<a_2<b_2<\cdots <a_n<b_n\le h
\ee
such that
\be
T =\bigcup_{i=1}^n [a_i,b_i].
\ee
So
\begin{eqnarray}
C_M&\le
&\sum_{i=1}^n\int_{a_i}^{b_i} 
(1 + (\tilde{z}'(t))^2)^{1/2}- (1 + (z_1-z_0)^2/h^2)^{1/2}\,dt\\
&=&\sum_{i=1}^n\int_{a_i}^{b_i} 
\left( \int_{|z_1-z_0|/h}^{|\tilde{z}'(t)|} 
\frac{d}{dy} \sqrt{y^2+1} \, dy \, \right) dt\\
&=&\sum_{i=1}^n\int_{a_i}^{b_i} 
\left( \int_{|z_1-z_0|/h}^{|\tilde{z}'(t)|} 
\frac{y}{ \sqrt{y^2+1}} \, dy \, \right) dt \\
&\le&\sum_{i=1}^n\int_{a_i}^{b_i} 
\left( \int_{|z_1-z_0|/h}^{|\tilde{z}'(t)|} 1 \, dy \, \right) dt \\
&=& \sum_{i=1}^n\int_{a_i}^{b_i}  |\tilde{z}'(t)|\,-\, |(z_1-z_0)/h| \, \,dt \\
&\le& \sum_{i=1}^n\int_{a_i}^{b_i}  |\tilde{z}'(t)|\,+\, |z_1-z_0|/h \, \,dt \\
&\le& \int_{0}^{h}  |\tilde{z}'(t)|\,+\, |z_1-z_0|/h \,\, dt \\
&=& \int_{0}^{h}  |\tilde{z}'(t)|\,dt \,\,+\,\,  |z_1-z_0|  
\end{eqnarray}
Since $\tilde{z}(0)=z_0$ and $\tilde{z}(h)=z_1$ we have
\begin{eqnarray}
C_M&\le & \int_{0}^{h}  |\tilde{z}'(t)|\,dt \,+\,\left| \int_{0}^{h} \tilde{z}'(t)\,dt \right|   \\
&\le & \int_{0}^{h}  |\tilde{z}'(t)|\,dt \,+\, \int_{0}^{h} |\tilde{z}'(t)|\,dt    \\
&=&  2 \int_0^h |\tilde{z}'(t)| \, dt. 
\end{eqnarray}
To obtain (\ref{eqn-C_M-bound-1}) 
we apply (\ref{eqn-est-C_M-7}) and the fact that $h\le \diam(W)$
from (\ref{eqn-est-C_M-15}).  
\end{proof}


\begin{rmrk}\label{rmrk-embed-const}
At the end of the proof we could have taken
a much more subtle estimate of $C_M$ as an integral
of $|\grad F|$ over a curve.  However this overestimate
suffices for our purposes.
\end{rmrk}

\section{Positive Mass Stability Theorem}
\label{Sect-Pos}

In this section we will prove Theorem~\ref{thm-main}
by constructing an explicit filling between the
two tubular neighborhoods, $T_D(\Sigma_{\alpha_0})\subset M^m$
and $T_D(\Sigma_{\alpha_0})\subset Z^m$.  
We have a Riemannian embedding of $M^m$
and $\E^m$
into $\E^{m+1}$ by Lemma~\ref{lem-graph} which we
can use to fill in the space between the tubular neighborhood in $M^m$
and its projection in $\E^m$.  To create a metric
isometric embedding we
will attach a strip by applying Theorem~\ref{thm-Z}
as in Figure~\ref{fig-flat-filling}. 

To define the filling manifold and excess boundary more
precisely, we recall the radial function:
\be
r(\Sigma_\alpha)= \left( \alpha /\omega_{m-1}\right)^{1/(m-1)}.
\ee
Setting 
\begin{eqnarray}
r_{min}&=& \inf\{r(p): \, r\in M^m\}\\
r_{D-}&=&\inf \{ r(p): p\in T_D(\Sigma_0) \subset M^m\}\\
r_0&=& r(\Sigma_{\alpha_0})=(\alpha_0/\omega_{m-1})^{1/(m-1)}\\
r_{D+}&=&\sup \{ r(p): p\in T_D(\Sigma_0) \subset M^m\}.
\end{eqnarray}
we see that $r_{min}\le r_{D-}\le r_{D+}$ all depend on the
manifold while $r_0$ is an invariant for Theorem~\ref{thm-main}.
Since
\be
r_0-D \le r_{D-}\le r_0 \le r_{D+} \le r_0+D \,\,\textrm{ and }\,\,0\le r_{min}
\ee
the tubular neighborhood in $M^m$ projects to
\be
r^{-1}(r_{D-}, r_{D+})\,\,
\subset\,\, T_D(\Sigma_0)\,\, \subset\,\, \E^m.
\ee

We will define a filling between the tubular neighborhood and
this projection.  See Figure~\ref{fig-flat-filling}.  The region
\be
A_0\,\,=\,\,Ann_0\left(r_{D+}, r_0+D\right)\,\,\subset\,\, T_D(\Sigma_{\alpha_0})\,\,\subset\,\, \E^m,
\ee
will form part of our excess boundary and its volume will
be estimated in Lemma~\ref{lem-switch-1}.  The inner
region is more complicated as their may be a deep well in $M^m$.

To avoid difficulties with deep wells, in Lemma~\ref{lem-well},
we will choose $r_\epsilon'\in (0,r_0)$ 
where $r_0=r(\Sigma_{\alpha_0})=(\alpha_0/\omega_{m-1})^{1/(m-1)}$ 
and cut off the
well 
\be\label{A-1}
A_1=r^{-1}(r_{D-},r_\epsilon)\subset T_D(\Sigma_{\alpha_0})\subset M^m
\ee
and the corresponding annulus
\be\label{A-2-1}
A_2=A_{2,1}=Ann_0(r_0-D,r_\epsilon) \subset T_D(\Sigma_{\alpha_0})\subset \E^m 
\ee
where 
\be \label{r-epsilon-prime}
r_\epsilon=\max\{r_\epsilon', r_{D-}\}
\ee
Note that when $r_0-D\le 0$, $A_2=B_0(r_\epsilon)$
as depicted in Figure~\ref{fig-flat-filling}.  The volumes of 
regions $A_1$ and $A_2$ are uniformly estimated in Lemma~\ref{lem-well}.

Note when $r_\epsilon'\le r_{D-}$ our tubular neighborhood
is not intersecting with a deep well and we have
$A_1=\emptyset$.  In that case we set
\be\label{A-2-2}
A_2=A_{2,2}=Ann_0(r_0-D,r_{D-}) \subset T_D(\Sigma_{\alpha_0})\subset \E^m.
\ee
The volume of $A_2$ is bounded uniformly in Lemma~\ref{lem-switch-2}.

Next we choose Riemannian isometric embeddings of 
\begin{equation*}
{r^{-1}(r_\epsilon, r_{D+})=
T_D(\Sigma_{\alpha_0})\setminus A_1
\subset M^m}
\textrm{ and }
{r^{-1}(r_\epsilon, r_{D+})= 
T_D(\Sigma_{\alpha_0})\setminus (A_0\cup A_2)
\subset \E^m}
\end{equation*}
into 
$r^{-1}(r_\epsilon, r_{D+})\subset \E^{m+1}$
such that $\Sigma_{\alpha_\epsilon}\subset M^m$
and $\Sigma_{\alpha_\epsilon}\subset \E^m$ coincide.
Lemma~\ref{lem-graph} determines this embedding up to
a vertical shift, so this is possible.

This determines the region:
\be \label{B_1}
B_1=\{(x_1,...x_m,z): \,
z\in [0,F(r)],\, r\in (r_\epsilon, r_{D+})\}\subset
r^{-1}(r_\epsilon,r_{D+})\subset \E^{m+1},
\ee
between these Riemannian isometric embeddings.
The region $B_1$ is not a filling manifold.
The Euclidean annulus has a metric isometric
embedding, but not the region in $M^m$.
So we add a strip 
\be \label{B_2}
B_2=[0,S_M]\times r^{-1}[r_\epsilon, r_{D+}]\subset [0,S_M]\times M
\ee
where width $S_M$ is determined in Lemma~\ref{lem-width}
using Theorems~\ref{thm-embed-const} and~\ref{thm-Z} and isometrically
embed 
\be
\textcolor{blue}{r^{-1}(r_\epsilon, r_{D+})\subset M^m} \textrm{ and } 
\textcolor{red}{r^{-1}(r_\epsilon, r_{D+})\subset \E^m}
\ee
into $B_1 \cup B_2$ as in Figure~\ref{fig-flat-filling}.

\begin{figure}[h] 
   \centering
   \includegraphics[width=3in]{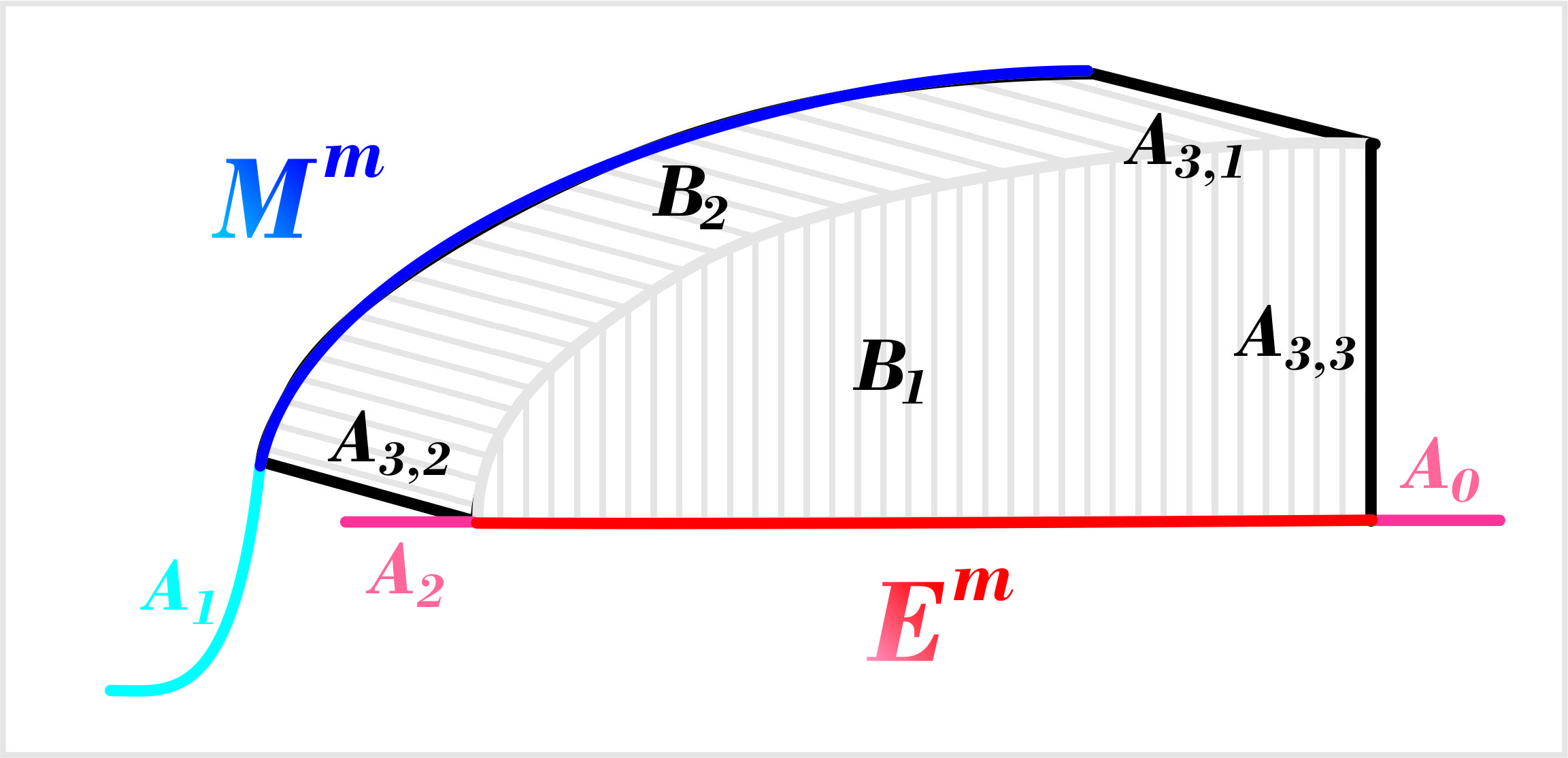} 
   \caption{Explicit Isometric Embedding into Z
   }
   \label{fig-flat-filling}
\end{figure}

Applying Proposition~\ref{embed-const-2}, 
we then know
\be \label{eqn-A_i-B_i}
d_{\mathcal{F}}(\,\textcolor{blue}{T_D(\Sigma_{\alpha_0})
\subset M^m}\,
,\,\textcolor{red}{T_D(\Sigma_{\alpha_0})\subset \E^m}
\,)\le \vol_{m+1}(B)
+\vol_m(A)
\ee
where $B=B_1\cup B_2$ is the filling manifold
and 
\be
A=A_0+A_1+A_2+A_{3,1}+A_{3,2} + A_{3,3} 
\ee
is the excess boundary with
\be\label{A-3-1}
A_{3,1}=[0,S_M]\times r^{-1}\{r_{D+}\}\subset [0,S_M]\times M,
\ee
\be\label{A-3-2}
A_{3,2}=[0,S_M]\times r^{-1}\{r_\epsilon\}\subset [0,S_M]\times M
\ee
\be\label{A-3-3}
A_{3,3}=r^{-1}\{r_{D+}\}\subset \partial B_1 \subset \E^{m+1}.
\ee
By estimating the volumes of these regions we will complete
the proof of Theorem~\ref{thm-main}

\subsection{Cutting off the Deep Wells}\label{ss-cut-well}

As seen in Figure~\ref{fig-thm-pted}, the manifolds with small mass
can have arbitrarily deep wells.  Here we determine
where to cut them off.

\begin{lem} \label{lem-well}
Given $\epsilon>0$, $D>0$ $\alpha_0>0$ and 
$M^m\in \RS$, and 
a symmetric sphere $\Sigma_{\alpha_0}\in M^m$ 
of area $\alpha_0$.
Let
\be
\alpha_\epsilon= \min\left\{\epsilon/ (16D),\,\, 
\omega_{m-1}^{1/m}(\epsilon/8)^{(m-1)/m},\,\, \alpha_0\right\}.
\ee
Choose 
\be
r'_\epsilon=r(\Sigma_{\alpha_\epsilon})=
(\alpha_\epsilon/\omega_{m-1})^{1/(m-1)}>0.
\ee
Then defining sets $A_1$ and $A_{2,1}$ in (\ref{A-1}) and
(\ref{A-2-1}) respectively, we have
\be
\vol\left(A_1\right)\le \epsilon/8, \,\vol\left(A_{2,1}\right)\le \epsilon/8, \textrm{ and }\vol\left(B_0(r_\epsilon')\subset \E^m\right)\le \epsilon/8.
\ee
\end{lem}

\begin{proof}
Since $B_0(r_\epsilon)\subset \E^m$
\be
\vol_m(B_0(r_\epsilon')) \le r_\epsilon' \alpha_\epsilon
\le \alpha_{\epsilon}(\alpha_\epsilon/\omega_{m-1})^{1/(m-1)} < \epsilon/8.
\ee
Now $A_{2,1}$ and $A_2$ are empty unless $r_\epsilon=r_\epsilon'$
so we assume this for the rest of the proof.   Then
$A_{2,1}\subset B_0(r_\epsilon)$ has volume $<\epsilon/8$ as well.

Let $z_\epsilon=z(\Sigma_{\alpha_\epsilon})$ and
$z_D=\min\{z(p): \, p\in T_D(\Sigma_{\alpha_0})\subset M^m\}$.
Observe that $z_\epsilon-z_D<D$
because we chose $\alpha_\epsilon< \alpha_0$
and areas are monotone in $\RS$ and $r^{-1}(r_{D-}, r_{D+})$
is in a tubular neighborhood of radius $D$ about $\Sigma$.
Observe that the cylinder
\be
C^m=\partial B_0(r_\epsilon) \times [z_D, z_\epsilon].
\ee
has volume 
\be
\vol_m(C^m)\le \alpha_\epsilon(z_\epsilon-z_D)\le
\alpha_\epsilon D\le \epsilon/16.
\ee
Since $F'(z)\ge 0$, we can project the well, 
$A_1\subset M^m$,
radially outwards to $C^m$ and vertically downwards to $B_{0}(r_\epsilon)$
to estimate the volume:
\be
\vol(A_1)\le \vol_M(C^m) +\vol_m(B_0(r_\epsilon)) <\epsilon/8.
\ee
\end{proof}

\subsection{First Key Restrictions on $\delta$ for Theorem~\ref{thm-main}}

The first basic estimate  follows immediately
from Lemma~\ref{lem-graph}, Lemma~\ref{lem-F'}
and Lemma~\ref{lem-rmin}:

\begin{lem} \label{lem-Q}
Given fixed $r_\epsilon>0$, $D>0$, $\alpha_0>0$, $m\in \N$,
choose 
\be
\delta< \delta(r_\epsilon)=(r_\epsilon/2)^{1/(m-2)}.
\ee
If $M^m\in \RS_m$ has $m_{ADM}<\delta$ then
then $M^m$ has a Riemannian isometric embedding into
\be
\{z=F(r)\}\subset \E^{m+1}
\ee
where $F:[r_{min},\infty) \to \R$ is
an increasing function, 
\be \label{rmin-delta}
r_{min}\le (2\delta)^{1/(m-2)}< r_\epsilon
\ee
and 
\be\label{eq-Q}
|F'(r)| \le Q(\delta,r_\epsilon)\qquad \forall r\ge 
r_{\epsilon},
\ee
where
\be
Q(\delta,r):=\sqrt{2\delta/(r_\epsilon^{m-2}-2\delta)}.
\ee
Observe that
\be
\lim_{\delta\to 0}Q(\delta,r_\epsilon)=0
\ee 
for fixed $r_\epsilon$.
\end{lem}

\begin{proof}
Lemma~\ref{lem-graph} provides the Riemannian
isometric embedding and Lemma~\ref{lem-rmin}
provides (\ref{rmin-delta}).
Lemma~\ref{lem-F'} and the fact that
$m_{ADM}(M)< 2\delta$ then implies that
\be
|F'(r)| \le Q(\delta,r):=\sqrt{2\delta/(r^{m-2}-2\delta)} 
\qquad \forall r\ge
(2\delta)^{1/(m-2)}.
\ee
By our choice of $\delta$, $r_\epsilon>(2\delta)^{m-2}$,
so we get (\ref{eq-Q}) by applying the fact that
$Q(\delta,r)$ decreases in $r$.
\end{proof}

Since we have already controlled all regions with $r<r_\epsilon$
in the last subsection, we can control the rest of the regions
by taking $\delta$ small enough that we can apply
Lemma~\ref{lem-Q}.

\subsection{From the Projected Set to Tubular Neighborhood}

Here we estimate the volumes of the regions
between the projected set $r^{-1}(r_{D-}, r_{D+})\subset \E^m$ 
and the tubular neighborhood
$T_D(\Sigma_{\alpha_0})\subset \E^m$ proving Lemma~\ref{lem-switch-1}
and Lemma~\ref{lem-switch-2}.

\begin{lem}\label{lem-switch-1}
Given $D>0$, $\alpha_0>0$, $m\in \N$,
Choosing $\delta>0$ such that
\be
(2\delta)^{m-2}<r_0/2.
\ee
If $M^m\in \RS_m$ has $m_{ADM}<\delta$ 
then
\be
\vol_m(A_0)\le D Q(\delta,r_0) \omega_{m-1}(r_0+D)^{m-1}\\
\ee
\end{lem}

\begin{proof}  
By our choice of $\delta$ we know that
\be
|F'(r)| \le Q(\delta, r_0) \qquad \forall r\ge r_0.
\ee
By the formula for arclength
\begin{eqnarray}
r_0+D-r_{D+}&=&r_0-r_{D+}+\int_{r_0}^{r_{D+}} \sqrt{1+F'(r)^2}\, dr\\
&\le & r_0-r_{D+}+(r_{D+}-r_0) (1+ Q(\delta,r_0))\\
&\le & (r_{D+}-r_0) Q(\delta,r_0)\le DQ(\delta,r_0)\\
\end{eqnarray}
Thus 
\begin{eqnarray}
\vol(A_0)&\le & \vol\left(r^{-1}(r_{D+}, r_0+D)\subset \E^m\right)\\
&\le & \left(r_{0}+D-r_{D+}\right) \omega_{m-1}(r_0+D)^{m-1}\\
&\le & D Q(\delta,r_0) \omega_{m-1}(r_0+D)^{m-1}
\end{eqnarray}
and the lemma follows.
\end{proof}

Recall that region $A_{2,2}$ defined in (\ref{A-2-2}) is only
defined when $r_\epsilon \le r_{D-}$.  So the next lemma,
estimating it's volume, assumes this condition.

\begin{lem}\label{lem-switch-2}   
Given $r_\epsilon>0$ $D>0$, $\alpha_0>0$, $m\in \N$,
we choose $\delta$ is in Lemma~\ref{lem-Q}.
If $M^m\in \RS_m$ has $m_{ADM}<\delta$ 
the region $A_{2,2}$ defined in (\ref{A-2-2}) satisfies
\be
\vol_m(A_{2,2}) \le D Q(\delta,r_\epsilon) \omega_{m-1}(r_0)^{m-1}.
\ee
\end{lem}

\begin{proof}
By our choice of $\delta$ we know that
\be
|F'(r)| \le Q(\delta, r_\epsilon) \qquad \forall r\ge r_{D-}\ge r_\epsilon.
\ee
By the formula for arclength
\begin{eqnarray}
r_{D-}-(r_0-D)&=&r_{D-}-r_{0}+\int_{r_{D-}}^{r_{0}} \sqrt{1+F'(r)^2}\, dr\\
&\le & r_{D-}-r_0+(r_{0}-r_{D-}) (1+ Q(\delta,r_\epsilon))\\
&\le & (r_{0}-r_{D-}) Q(\delta,r_0)\le DQ(\delta,r_\epsilon)\\
\end{eqnarray}
Thus 
\begin{eqnarray}
\vol(A_{2,2})&= & \vol(r^{-1}(r_0-D, r_{D-})\subset \E^m)\\
&\le & D Q(\delta,r_\epsilon) \omega_{m-1}(r_0)^{m-1}
\end{eqnarray}
and the lemma follows.
\end{proof}

\subsection{Choosing the Width of the Strip}
\label{ss-choose}

\begin{lem}\label{lem-width}
Given $r_\epsilon>0$ $D>0$, $\alpha_0>0$, $m\in \N$,
we choose $\delta$ is in Lemma~\ref{lem-Q}.
If $M^m\in \RS$ and
$\mathrm{m}_{\mathrm{ADM}}(M^m) < \delta$,
then the region $r^{-1}(r_\epsilon,r_{D+})$ has a metric isometric embedding
into the filling
manifold $B_1\cup B_2$ of (\ref{B_1})
and (\ref{B_2})  
where
\be
S_M=S(\delta,r_\epsilon, D, r_0)
= \sqrt{C (2D +\pi r_0 + C)}
\ee
with
\be
C=C(D, r_0, \delta, r_\epsilon)= (4D+2\pi r_0) Q(\delta, r_\epsilon)
\ee

\end{lem}

\begin{proof}
We begin by applying Theorem~\ref{thm-embed-const}, to
the Riemannian isometric embedding of
$r^{-1}(r_\epsilon, r_{D+})\subset M^m$
into $W \times \R \subset \E^{m+1}$
where $W=Ann_0(r_\epsilon, r_{D+})\subset \E^m$.

Since $r^{-1}(r_\epsilon, r_{D+})\subset T_D(\Sigma_{\alpha_0})$,
we have
\be
\diam(W)\le\diam\left(r^{-1}(r_\epsilon, r_{D+})\right) \le 2D + \diam\left(\Sigma_{\alpha_0}\right)=2D+\pi r_0.
\ee
By Lemma~\ref{lem-Q} have a bound on $F'$ which gives
us an embedding constant $C_M$ which is less than  $C$
given above.
By Theorem~\ref{thm-Z}, the strip width, 
$S_M$ given above suffices to obtain a metric isometric embedding.
\end{proof}

\subsection{Volume Estimates and the Proof of Theorem~\ref{thm-main}}
\label{ss-vol}

The proof of Theorem~\ref{thm-main}
is completed by estimating the volumes of the
regions depicted in Figure~\ref{fig-flat-filling}
and applying (\ref{eqn-A_i-B_i}):

\begin{proof}
Given any $\epsilon>0$, $D>0$, $\alpha_0>0$, $m\in \N$
we choose 
\be
r_\epsilon'>0
\ee 
depending only on $\epsilon$ and $\alpha_0$
exactly as in Lemma~\ref{lem-well}.   We set
$r_0>0$ such that $\alpha_0=\omega_{m-1}r_0^{m-1}$.

We choose 
\be\label{choose-delta-1}
\delta<\delta(r_\epsilon)
\ee 
as in Lemma~\ref{lem-Q}.   We will refine it further later
in (\ref{choose-delta-1}),
(\ref{choose-delta-2}),
(\ref{choose-delta-3}),
(\ref{choose-delta-4}),
(\ref{choose-delta-5})
and (\ref{choose-delta-6}) to obtain
$\delta=\delta(\epsilon, D, \alpha_0, m)>0$.

Assume $M^m\in\RS_m$ has ADM mass 
$\madm(M)<\delta$.

We set 
\be
r_\epsilon=\max\{r_\epsilon', r_{D-}\}.
\ee
as in (\ref{r-epsilon-prime}).
When $r_\epsilon\ge r_{D-}$ we apply Lemma~\ref{lem-well},
(\ref{A-1}) and
(\ref{A-2-1}) to see that 
\be
\vol_m(A_1)+\vol_m(A_2)\le \epsilon/8+\epsilon/8=\epsilon/4.
\ee
When $r_\epsilon< r_{D-}$, then
$A_1=\emptyset$ and by Lemma~\ref{lem-switch-2}, we
obtain the same estimate as long as
we choose $\delta>0$ is chosen small enough that:
\be\label{choose-delta-2}
DQ(\delta, r_\epsilon)\omega_{m-1}(r_0)^{m-1}< \epsilon/8.
\ee
This second restriction on $\delta$ also suffices to
obtain
\be
\vol_m(A_0)<\epsilon/8.
\ee

By Lemma~\ref{lem-graph} we have a Riemannian
isometric embedding of $r^{-1}(r_\epsilon,r_{D+})$
into $\{z=F(r)\}\subset \E^{m+1}$ and may
define $B_1$ as in (\ref{B_1}).  We then have
\begin{eqnarray}
\vol_{m+1}(B_1)
&=& \int_{r_\epsilon}^{r_{D+}} (F(r)-F(r_\epsilon)) \omega_{m-1}r^{m-1} \, dr\\
&\le & (r_{D+}-r_\epsilon) \omega_{m-1}r_{D+}^{m-1} 
(F(r_{D+})-F(r_\epsilon))\\
&\le & 2D \omega_{m-1}(r_0+D)^{m-1} 
\int_{r_\epsilon}^{r_{D+}} F'(r)\, dr < \epsilon/8.
\end{eqnarray}
as long as $\delta$ is chosen small enough that
\be \label{choose-delta-3}
4D^2 \omega_{m-1}(r_0+D)^{m-1}Q(r_\epsilon, \delta) < \epsilon/8.
\ee
Applying Lemma~\ref{lem-width} to
create region $B_2$ as in (\ref{B_2}) such that
\begin{eqnarray} 
\vol_{m+1}(B_2)
&=&  S_M \vol(r^{-1}(r_\epsilon, r_{D+})\subset M^m) \\
&=& 
S_M \int_{r_\epsilon}^{r_{D+}} \sqrt{1+F'(r)^2}\, \omega_{m-1}r^{m-1} \, dr\\
&=& 
S_M \int_{r_\epsilon}^{r_{D+}} \left(1+F'(r)\right)\, \omega_{m-1}r^{m-1} \, dr
<\frac{\epsilon}{8}
\end{eqnarray}
as long as $\delta$ is chosen small enough that
\be \label{choose-delta-4}
S(\delta,r_\epsilon, D, r_0) 2D \omega_{m-1}(r_0+D)^{m-1} Q(\delta, r_\epsilon)<\epsilon/8.
\ee
By (\ref{A-3-1}), (\ref{A-3-2})  we have
\begin{eqnarray}
\vol_m(A_{3,1})&=& S_M \omega_{m-1}r_{D+}^{m-1} \le S_M
\omega_{m-1}(r_0+D)^{m-1} < \epsilon/12\\
\vol_m(A_{3,2})&=& S_M \omega_{m-1}r_\epsilon^{m-1} \le S_M \omega_{m-1}(r_0)^{m-1} <\epsilon/12
\end{eqnarray}
as long as $\delta$ is chosen small enough that
\be \label{choose-delta-5}
S(\delta, r_\epsilon, D, r_0)\omega_{m-1}(r_0+D)^{m-1} < \epsilon/12
\ee
By (\ref{A-3-3}) we have
\begin{eqnarray}
\vol_m(A_{3,3})&=& \omega_{m-1}r_{D+}^{m-1}(F(r_{D+})-F(r_\epsilon))\\
&\le& \omega_{m-1}(r_0+D)^{m-1} Q(\delta, r_\epsilon)
< \epsilon/12
\end{eqnarray}
as long as $\delta$ is chosen small enough that
\be \label{choose-delta-6}
\omega_{m-1}(r_0+D)^{m-1} Q(\delta, r_\epsilon)
< \epsilon/12.
\ee
The theorem follows from (\ref{eqn-A_i-B_i})
summing over all these volumes.
\end{proof}

\begin{rmrk} \label{rmrk-scaling}
Note that we have linear scaling on
\be
m(A)^{1/m}+m(B)^{1/{m+1}}
\ee
in this proof.  
Redefining the Intrinsic Flat Distance in this way might
be worth investigating as it is apparently still a distance.  Such a redefinition appears to
induce the same intrinsic flat topology on the
space of Riemannian manifolds.   Recall that the
flat distance was originally defined by
Federer-Fleming \cite{FF} to be a norm: linear in
multiplication of an integral current by a magnitude.
This property should only be abandoned with caution. 
See Remark~\ref{rmrk-scaling-open}.
\end{rmrk}

\section{Gromov-Hausdorff Distance}\label{Sect-Gromov}

In this section we review the Gromov-Hausdorff distance
between Riemannian manifolds, provide new
estimates for estimating the Gromov-Hausdorf distance
[Propositions~\ref{embed-const-3} and~\ref{embed-const-4}]
based on the embedding constants defined in 
Theorem~\ref{thm-embed-const} and then prove in Example~\ref{ex-not-GH}
that the Positive Mass Theorem is not stable with respect to the
Gromov-Hausdorff distance.

Recall that the Gromov-Hausdorff distance was first defined by Gromov
in \cite{Gromov-metric} as follows:
\be
d_{GH}(M_1^m, M_2^m)=\inf_{\varphi_i: M_i \to Z} d_H^Z
\left(\varphi_1(M_1), \varphi_2(M_2)\right),
\ee
where the infimum is taken over all metric spaces $Z$ and all
metric isometric embeddings $\varphi_i: M_i \to Z$ and where
the Hausdorff distance in $Z$ between two subsets $X_1$ and
$X_2$ is
\be
d_H(X_1, X_2) =\inf\{\rho>0: \,\, X_1 \subset T_\rho(X_2)
\textrm{ and } X_2\subset T_\rho(X_1).
\ee

\subsection{New Estimates using Embedding Constants}

Naturally the techniques given in Section~\ref{Sect-Embed}
may also be applied to estimate the Gromov-Hausdorff
distance.  In particular we see that Theorem~\ref{thm-embed-const}
implies the following proposition 
much as it implies Proposition~\ref{embed-const}:

\begin{prop}\label{embed-const-3}
If $M^m_i$ are Riemannian manifolds have
Riemannian isometric embeddings $\varphi_i: M^m_i \to N^{m+1}$ with 
embedding constants $C_{M_i}$ as in (\ref{eqn-embed-const-1}), and if
their images are disjoint and lie in the boundary of a region $B_0 \subset N$
then
\begin{eqnarray}
d_{GH}(M_1, M_2) &\le& 
S_{M_1}+S_{M_2} + d_H^N(\varphi_1(M_1), \varphi_2(M_2))
\end{eqnarray}
where $S_{M_i}=\sqrt{C_{M_1}(\diam(M_i)+C_{M_i})}$. 
\end{prop}

Notice how the Gromov-Hausdorff distance does not allow one to cut
off a well using only its volume to estimate it.  
The depth of the well will contribute to the distance.  
In place of Proposition~\ref{embed-const-2} we have:

\begin{prop}\label{embed-const-4}
If $M^m_i$ are Riemannian manifolds and $U^m_i\subset M^m_i$
are submanifolds that have
Riemannian isometric embeddings $\varphi_i: U^m_i \to N^{m+1}$ with 
embedding constants $C_{U_i}$ as in (\ref{eqn-embed-const-1}), and if
their images are disjoint and lie in the boundary of a region $B_0 \subset N$
then
\begin{eqnarray}
d_{GH}(M_1, M_2) &\le& 
S_{U_1}+S_{U_2} + d_H^N(\varphi_1(U_1), \varphi_2(U_2))\\
&&\qquad +\sup_{x\in M_1\setminus U_1} d_{U_1}(x, M_1)
+\sup_{x\in M_2\setminus U_2} d_{U_2}(x, M_2)
\end{eqnarray}
where $S_{M_i}=\sqrt{C_{M_1}(\diam(M_i)+C_{M_i})}$. 
\end{prop}

If the regions $M\setminus U$ is a deep well, then Proposition~\ref{embed-const-3} will provide a very poor estimate for the Gromov-Hausdorff
distance between the spaces.  In fact the spaces need not be close at all.

\subsection{Wells of Arbitrary Depth}

The Positive Mass Theorem is not stable with respect to the
Gromov-Hausdorff distance.  In fact,
the manifolds given in Example~\ref{ex-deep-well} are close
to Euclidean space with a line segment attached to it,
$\E^m \cup [0,L]$ where the
line segment can have arbitrary length.  That is:

\begin{ex}\label{ex-not-GH}
Given any $L_0>0$  there exists a sequence  
$M_j^m\in \RS_m$ such that 
$\lim_{j\to\infty}m_{\mathrm{ADM}}(M_j^m)=0$
and the pointed Gromov-Hausdorff limit of 
$M_j^m$ is $\E^m\disjointunion[0,L_0]$ in the sense that
for any fixed $\alpha_0>0$
and  any $D>0$ there exists $D_j \to D$ such that
\be
d_{GH}\left(B_{p_j}(D_j)\subset M^m, \,\,B_0(D) \subset \E^m \disjointunion [0,L_0]\right) < \epsilon
\ee
where $p_j \in \Sigma_{\alpha_0} \subset M_j$.
\end{ex}

\begin{proof}
Let $r_0=(\alpha_0/\omega_m)^{1/(m-1)}$.
Take $\delta_j=1/j$ and
and take $M_j$ to be the manifold in Example~\ref{ex-deep-well}
with 
\be 
L=d_M(\Sigma_{\alpha_0}, \Sigma_{min})=L_0+r_0.
\ee
and $m_{ADM}(M_j) <\delta_j$.   

Now fix $D>0$.  We isometrically embed the tubular neighborhoods
$T_D(\Sigma_{\alpha_0})\subset M_j$ and $T_D(\Sigma_{\alpha_o})\subset \E^m$ into $Z$ as in the proof of Theorem~\ref{thm-main}.  We
attach $\{0\}\in [0,L_0]$ to Euclidean space at a point in
$\Sigma_{r_\epsilon}$ and then the interval runs down the well in
$M_j$ to the bottom of the well at $L_0+r_0-r_\epsilon$
and a little further a distance $r_\epsilon$ as an extra
segment.   We extend $Z$ as well.   
 
Thus $M_j \subset T_{\rho_j}(\E^m \cup [0, L_0])$ where
\be
\rho_j = \max \left\{ F(r_D)-F(r_\epsilon)+S_{M_j}, \pi r_\epsilon \right\}.
\ee
On the other hand $\E^m\disjointunion [0, L_0]\subset T_{\rho_j'}(M_j)$
where
\be
\rho_j'=\max \left\{ r_\epsilon, F(r_D)-F(r_\epsilon)+S_{M_j} \right\}
\ee
since $d_{\E^m}(0, \partial B_p(r_\epsilon))=r_\epsilon$
and the segment has extra length $r_\epsilon$.
Thus
\begin{eqnarray}
d_{GH}\left(T_D(\Sigma_{\alpha_0})\subset M_j^m, \,\,
T_D(\Sigma_{\alpha_0})\subset \E^m\disjointunion[0,L_0]\right)\qquad&&\\
\qquad\le \,\,d_{H}^Z
\left(T_D(\Sigma_{\alpha_0})\subset M_j^m, \,\,
T_D(\Sigma_{\alpha_0})\subset \E^m\disjointunion[0,L_0]\right)
&<& \max \left\{\rho_j, \rho'_j\right\}.
\end{eqnarray}
As $\delta_j \to 0$, we have $r_\epsilon \to 0$ and 
$S_{M_j}\to 0$ in the proof of Theorem~\ref{thm-main} 
so $\rho_j, \rho_j' \to 0$.

Since this is true for all $D$, we can exhaust the
space with the tubular neighborhoods and obtain
the pointed Gromov-Hausdorff onvergence.
\end{proof}

\subsection{Arbitrarily Dense Collections of Wells}
\label{ss-ex2}

If we remove the requirement that a manifold be rotationally
symmetric then we can introduce more than one well.  In
fact we can create asymptotically flat manifolds, $M_j^m$ with
positive scalar curvature and $M_{ADM}(M_j) \to 0$ that have
increasingly dense collections of wells [Example~\ref{ex-noncompact}].
Such a sequence of manifolds doesn't even have a subsequence
converging in the Gromov-Hausdorff sense.

These examples are
based upon the work of Schoen-Yau and Gromov-Lawson \cite{Schoen-Yau-tunnels}\cite{Gromov-Lawson-tunnels}
who have proven that if one has a manifold of constant sectional
curvature, then one can attach a well of arbitrary depth and
thinness to that manifold maintaining positive scalar curvature.
For this reason we are limiting ourselves to dimension three
in the construction of the example, however, similar examples
should exist in higher dimensions as well.

We begin by creating an element of $\RS_m$ with stripes
of positive sectional curvature.

Recall that by Lemma~\ref{lem-ex} we need only create an admissible
Hawking function with the desired properties to produce
an element of $\RS_m$.

\begin{lem} \label{lem-stripe}
Let $M^m\in \RS$ and $K>0$.   $M$ has constant sectional curvature, $K>0$,
on $r^{-1}(a,b)\subset M$ iff $r^{-1}(a,b)\subset M$ is an annulus in a sphere of
radius $1/K^{1/2}$ iff $\mathrm{m}_{\mathrm{H}}(r)=r^mK/2$ for $r \in (a,b)$. 
\end{lem}

\begin{proof}
If it is an annulus in a sphere, then $(z-\zeta)^2+r^2=1/K$, so
$2(z-\zeta) z' +2r=0$ and thus $z'=-r/(z-\zeta)$ and 
by Lemma~\ref{lem-graph}
\be
\mathrm{m}_{\mathrm{H}}(r)=\frac{r^{m-2}}{2}\frac{r^2/(z-\zeta)^2}{1+r^2/(z-\zeta)^2} = \frac{r^m}{2(r^2+(z-\zeta)^2)}=\frac{r^mK}{2}.
\ee
On the other hand, if $\mathrm{m}_{\mathrm{H}}(r)=r^3K/2$, then 
(\ref{lem-ex-m-H-to-z}) defines a function $z(r)$ uniquely
up to a constant.  Since $z(r)=\sqrt{(1/K-r^2} + \zeta$
satisfies the equation, the graph is an annulus in a sphere of radius $1/K^{1/2}$.
\end{proof}

\begin{example}\label{ex-stripes}
Fix  $\delta>0$.
Given any increasing sequence, 
\be
\{r_1, r_2,...\}\subset [\mathrm{m}_{\mathrm{fix}}/2,\infty),
\ee
there exists $M^3\in \RS_3$ with constant
sectional curvature on stripes $r^{-1}(a_j,b_j)$
where $(a_j, b_j)\subset [r_{2j-1}, r_{2j}]$
and $\mathrm{m}_{\mathrm{ADM}}(M)<\delta$ 
and $\partial M=\emptyset$.
\end{example}

\begin{proof}
Recall that an admissable Hawking function need only be increasing
and satisfy
\be
\mathrm{m}_{\mathrm{H}}(r) \le h(r):=\min\{r/2, \mathrm{m}_{\mathrm{ADM}}\}.
\ee
For each $j$, choose the sectional curvature for the $j^{th}$
annulus to be $K_j$ satisfying
$
r_{2j}^3K_j /2 =h(r_{2j}).   
$
Observe that $K_j$ is a decreasing sequence and 
\be
r^3 K_{j+1}/2 < r^3K_j/2 < h(r) \textrm{ for } r < r_{2j}.
\ee
We now define $a_j< b_j$ inductively.  Let $a_1=r_1$.
So $a_1^3K_1/2< h(a_1)$.

Next choose $b_j\in (a_j,(a_j+ r_{2j})/2)$ satisfying
\be
b_j^3 K_j/2 \le  (a_j^3 K_1/2 + h(b_j))/2 < h(b_j).
\ee
Finally choose $a_{j+1} \in (r_{2j+1}, r_{2j+2})$ satisfying
\be
a_{j+1}^3 K_{j+1}/2 \in (b_j^3 K_j/2, h(a_{j+1}) )
\ee
which exists by our choice of $K_j$.  We can choose
any smooth increasing function 
$\mathrm{m}_{\mathrm{H}}: [0, \infty) \to [0, \delta)$
such that
\be
\mathrm{m}_{\mathrm{H}}(r)= r^3 K_j/2 \textrm{ for } r \in [a_j,b_j]
\ee
and then apply Lemma~\ref{lem-stripe} and~\ref{lem-ex}. 
\end{proof}

\begin{example} \label{ex-noncompact}
There exists a sequence of asymptotically flat manifolds
$M_i^3$ with no interior minimal surfaces and empty boundary
and $\lim_{i\to\infty} \mathrm{m}_{\mathrm{ADM}}(M_i) = 0$ 
such that for any $\alpha_0, D>0$
the sequence of regions $T_D(\Sigma)\subset M_i$
where $\vol_{2}(\Sigma)=\alpha_0$ converge in the
intrinsic flat sense to $T_D(\Sigma)\subset
\E^m$
but do not even have Lipschitz or Gromov-Hausdorff
converging subsequences.   
\end{example}


Recall that Gromov's Compactness Theorem states that
a sequence of compact metric spaces $X_j$
has a subsequence converging
to a compact metric space $X$ if and only if there is
a uniform bound on the number of disjoint balls of any
given radius in the space \cite{Gromov-metric}.  In particular,
a sequence of pointed  Riemannian manifolds $(M_j,p_j)$
has no subsequence
converging in the pointed Gromov-Hausdorff sense if there
is no uniform bound on the number $N(r,R)$ of
disjoint balls of radius $r$ lying in $B_{p_j}(R)$.
Here we will construct such a sequence of $M_j$ by 
gluing in increasingly many thin deep wells 
each of which contains a ball of radius $r$.

\begin{proof}
Fix $i \in \N$, $\delta=1/i$,
 and choose a sequence $r_j=  j / i$.  
 Then by Lemma~\ref{lem-stripe},
there exists $\bar{M}^3\subset \RS_3$ with $\mathrm{m}_{\mathrm{ADM}}
(M^3)=1/i$
that has stripes of
constant sectional curvature on annular regions
\be
r^{-1}(a_j,b_j)\subset r^{-1}[(2j-1)/i, {2j}/i].
\ee
By Schoen-Yau and Gromov-Lawson \cite{Schoen-Yau-tunnels}
\cite{Gromov-Lawson-tunnels}, we can remove
arbitrarily small balls, $B_{q_j}(\rho_j)\subset r^{-1}(a_j,b_j)$
for $j=1$ to $(2i)^2$ and 
attach arbitrarily thin and deep wells, $W_j$, to each of
these annular regions while maintaining nonnegative scalar
curvature and without changing the metric outside the 
removed balls.   In particular we can ensure that
all the attached wells, $W_j$, have a depth 
\be
\max\{d(x, \partial W_j): \, x\in W_j\}=2D
\ee
and we can ensure that 
\be
\sum_{j=1}^{(2i)^2} \left(\vol_m(W_j) +\vol_{m=1}(W_j) + 
\vol_m(B_{q_j}(\rho_j)) + \vol_m(\partial B_{q_j}(\rho_j)) \right) < 1/i.
\ee
and
$\diam(\partial B_{q_j}(\rho_j)) < d_i$.   We can also
require that all wells satisfy 
\be \label{where-wells}
W_j \subset T_D(\Sigma)\setminus \left(T_{D/10}(\Sigma)
\cup T_{D/10}(\partial \bar{M}_j)\right).
\ee
This gives us a non-rotationally symmetric manifold
$M^3$ which is asymptotically flat with $\mathrm{m}_{\mathrm{ADM}}(M)=\delta$
such that for $\Sigma=r^{-1}(s_0)$ where $s_0$ is rational we have
\be
d_{\mathcal{F}}(T_D(\Sigma)\subset M^3, 
T_D(\Sigma)\subset\bar{M}^3) < 1/i + \sqrt{C_i(C_i+2D+\pi r(\Sigma))}
\ee
where $C_i$ is the embedding constant of 
\be
\varphi_i: M^3\setminus \bigcup W_j \to \bar{M}^3.
\ee
We may choose $d_i$ sufficiently small to guarantee 
$\lim_{i\to\infty} C_i \to 0$.  Applying Theorem~\ref{thm-main}
to $\bar{M}^3$ we have the claimed intrinsic flat convergence.

On the other hand for fixed $s_0$, and increasing $i$
we have increasingly many wells contained in$T_D(\Sigma)\subset M_i$.
Since each well has depth $2D$, the boundary of$T_D(\Sigma)\subset M_i$
has increasingly many components.  So clearly we do not have Lipschitz
convergence to$T_D(\Sigma) \subset m_{\mathrm{Sch}}$ even if
we take a subsequence.

Also observe that if $\partial W_j\subset T_{D/3}(\Sigma)\subset M_i$
then
\be
d_{M_i}\left(W_j \cap \partial T_{D}(\Sigma), \partial W_j\right) > D/3.
\ee
Thus balls of radius $D/3$ about
$p_j \in W_j\cap \partial T_{D/3}(\Sigma)$ 
are pairwise disjoint and contained in $T_D(\Sigma)$.  
So we have increasing number of
pairwise disjoint balls centered in $T_{D}(\Sigma)\subset M_i$
 and thus$T_D(\Sigma)\subset M_i$ have no 
 subsequences converging in the Gromov-Hausdorff sense 
 \cite{Gromov-metric}.

\end{proof}

\section{Conjectures and Open Problems}
\label{Sect-Open}

We now consider the general case of complete asymptotically 
flat manifolds with nonnegative scalar curvature.  We will restrict 
to dimension three, because we have the most tools available in 
dimension three.     We consider whether the Positive Mass Theorem
is stable in this setting:

\begin{defn}
Let $\mathcal{M}$ be a subclass of asymptotically flat
three dimensional Riemannian manifolds with nonnegative
scalar curvature and no interior closed minimal surfaces
and either no boundary or the boundary is an outermost
minimizing surface.
\end{defn}

\begin{conj}\label{conj-main}   
Given any $\epsilon>0$, $D>0$, $\alpha_0>0$, 
there exists a $\delta=\delta(\epsilon, D, \alpha_0)>0$
such that if $M^3\in\mathcal{M}$ has ADM mass 
$\madm(M)<\delta$ and $\E^3$ is Euclidean space.
Then
\be
 d_{\mathcal{F}}\left(\,T_D(\Sigma_{\alpha_0})\subset M^3\,,\,T_D(\Sigma_{\alpha_0})\subset \E^3\,\right)\,<\, 
 \epsilon. 
 \ee
 where  $\Sigma_{\alpha_0}$ is a {\em special surface} of area  
$\vol_{2}(\Sigma_{\alpha_0})=\alpha_0$, and
$T_D(\Sigma_{\alpha_0})$ is the 
tubular neighborhood of radius $D$ around $\Sigma_{\alpha_0}$.   
 \end{conj}

We are deliberately vague as to the strength of our 
condition of {\em asymptotical flatness} in the definition
of $\mathcal{M}$.  The conjectures may
require strong conditions at infinity.
We have also been vague as to what the {\em special surface}, $\Sigma$,
should be.   We know the special surface must somehow
avoid wells but also be uniquely defined in Euclidean
space up to isometry.   We provide possible choices for
the strength of the asymptotic flatness and special surface in
the following remarks.


\begin{rmrk}{\em
Another possible choice of special surface, $\Sigma$, is
a Constant Mean Curvature surface.  One could say $\Sigma$ achieves an
isoperimetric condition: the surface enclosing the maximal
volume for its given area $\alpha_0$.   Note in Bray's thesis
it is proven that such a $\Sigma$ exists if it is connected \cite{Bray-thesis}.
One could for example assume that the
manifold has a smooth CMC foliation  down to $\Sigma$
with area $\alpha_0$.  Or one could just assume a smooth CMC
foliation exists  on $T_D(\Sigma)$ where $\Sigma$ is
a leaf  in the foliation with no such strong assumption at infinity.
There has been significant work on the existence of CMC foliations
and their properties beginning with Huisken-Yau \cite{Huisken-Yau}.
}\end{rmrk}

\begin{rmrk}{\em
A stronger condition on $\Sigma$ which might be viewed as 
a test case for the prior remark
would be to require positive Gauss curvature
or possibly even  lying in a foliation of such surfaces.   Nirenberg
proved that such $\Sigma$ isometrically embed into Euclidean space 
\cite{Nirenberg-surface} which we have shown provides a metric
isometric embedding in Theorem~\ref{thm-convex}.    Such surfaces have
a well defined quasi-local mass defined by Liu-Yau \cite{Liu-Yau-2003} 
based on work of Shi-Tam \cite{Shi-Tam}
which would be controlled by the ADM mass at infinity.}
\end{rmrk}

\begin{rmrk}\label{rmrk-smooth-IMCF}{\em
A possible choice of special surface, $\Sigma$, is
that it be a level set of Inverse Mean Curvature Flow from a point
or from the boundary of $M$.   One might assume
the manifold has a smooth IMCF in the conjecture or one might
assume only that the IMCF is smooth on a neighborhood containing
$T_D(\Sigma)$.   Geroch proved that smooth IMCF has a monotone
Hawking mass \cite{Geroch-monotonicity}, so it should be possible to control the metric in
a way somewhat similar to the way in which we applied
monotonicty of the Hawking mass to provide Lipschitz controls 
on our rotationally symmetric metrics.
}\end{rmrk}

\begin{rmrk}{\em
Huisken-Ilmanen extended the IMCF using Geometric Measure Theory
to prove the Penrose Conjecture (and reprove the Positive Mass Theorem)
\cite{Huisken-Ilmanen}.
Their proof uses a weak Inverse Mean Curvature Flow with a monotone
quasilocal mass.   Conjecture~\ref{conj-main}
might hold on any manifold satisfying the conditions of their theorem
where $\Sigma$ is a level set of their flow.  Many difficulties would
arise when trying to prove this.   Since
weak IMCF jumps over regions likes wells, one would need to
control the volumes of those regions separately.
}\end{rmrk}

\begin{rmrk}
One might consider the case where $M^3$ is a Spin manifold and apply
the work of Finster
\cite{Finster}.
Finster bounds the areas of level sets of spinors and controls the
$L^2$ norms of the curvature
tensor.  It is possible that level sets of spinors provide an
appropriate choice for the special
surface $\Sigma_0$ although we have not investigated this closely.
\end{rmrk}

\begin{rmrk}{\em
One might consider the case where $M^m$ is a graph in
Euclidean space.  Here one could examine the situation with many
wells and explicitly cut them out.  One could apply Theorem~\ref{thm-Z}
directly to find a filling manifold.   In the graph setting one might
test out various conditions at infinity and choices of special surface
$\Sigma$ perhaps even using numerical methods to solve IMCF
and find CMC surfaces.   Lam has provided a new short proof
of the Positive Mass Theorem in the graph setting which may
prove useful to those attempting to prove the conjecture in this
case \cite{Lam-graph}.
}\end{rmrk}

One may also consider the stability of the Penrose Inequality.
The authors have completed an investigation of this in
\cite{LeeSormani2}.  In fact, the Penrose Inequality is not
even stable in the rotationaly symmetric case.  However
sequences of manifolds approaching equality in the
Penrose inequality do have subsequences which
converge in the pointed intrinsic flat sense to manifolds
which are Schwarscshild spaces outside their outermost
minimal surface.  In fact, far stronger convergence can be
obtained as there are no thin central wells just
deep horizon central horizons which the authors prove converge
to cylinders of various lengths in the Lipschitz sense.   
The authors also prove Lipschitz convergence outside of the
central well in the Positive Mass setting in that paper.  
Without rotational symmetry, the authors provide an 
example with increasingly dense thin deep wells much like the example in this paper \cite{LeeSormani2}.
Thus one expects at best pointed intrinsic flat convergence without
rotational symmetry for almost equality of the Penrose Equality.

We close this paper  with a call for the investigation of a scalable
version of the Intrinsic Flat Distance.    

\begin{rmrk}\label{rmrk-scaling-open}{\em
Recall that the Intrinsic Flat Distance is the sum of a volume and an
area in (\ref{eqn-def-intrinsic-flat-1}).  This is a consequence of the
fact that the Intrinsic Flat Distance defined in \cite{SorWen2}
is based on the flat distance of Federer-Fleming \cite{FF} which is
a norm:
\be
d_{F}(T_1,T_2)=
|T_1 - T_2|_{\flat}= \inf \left\{M_{m}(A)+M_{m+1}(B): A+\partial B=T_1-T_2\right\}.
\ee 
One may immediately consider a related scalable intrinsic flat
distance which abandons the norm properties in favor of scalability
so that
\be
d_{sF}(T_1, T_2)= \inf \left\{M_{m}(A)^{1/m}+M_{m+1}(B)^{1/(m+1)}: 
A+\partial B=T_1-T_2\right\}.
\ee 
This is still a distance since it is nonnegative, symmetric, satisfies
the triangle inequality and 
\be
d_{sF}(T_1, T_2)=0 \iff d_{F}(T_1,T_2)=0 \iff T_1=T_2.
\ee
This can be seen by taking $A_i, B_i$ with $A_i+\partial B_i=T_1-T_2$
approaching the infimum and observing that $M(A_i), M(B_i) \to 0$
since masses of integral currents are nonnegative.

Thus one might consider defining an intrinsic scalable flat 
distance, $d_{s\mathcal{F}}$ between Riemannian manifolds
such that
\be
d_{s\mathcal{F}}(M^m_1, M^m_2) \le 
\vol_{m+1}\left(B^{m+1}\right)^{1/(m+1)} 
+\vol_m\left(A^m\right)^{1/m}
\ee
much as in \cite{SorWen2} and investigating which theorems
hold as they stand and which need adapting.  This investigation
would involve looking deeper than just this paper as the norm
properties were applied on more than one occasion and in 
citations.   

See Remark~\ref{rmrk-scaling} for information about estimating
this scalable flat distance in the rotationally symmetric case
of the almost inequality in the Positive Mass Theorem.  
}\end{rmrk}

While this final remark suggests a problem which would involve a 
strong understanding of geometric measure theory, we believe
other problems suggested in this paper are really questions of
geometric analysis.

\bibliographystyle{plain}
\bibliography{2011}

\end{document}